\documentclass[12pt, a4paper]{amsart}
\usepackage{amsmath}
\usepackage{amsxtra}
\usepackage{amscd}
\usepackage{amsthm}
\usepackage{amsfonts}
\usepackage{amssymb}
\usepackage {pstricks}
\usepackage{pstricks,pst-node}

\newtheorem{theorem}{Theorem}[section]
\newtheorem*{theorem*}{Theorem}
\newtheorem{lemma}[theorem]{Lemma}
\newtheorem{proposition}[theorem]{Proposition}
\newtheorem*{proposition*}{Proposition}
\newtheorem{corollary}[theorem]{Corollary}
\newtheorem{definition}[theorem]{Definition}

\theoremstyle{definition}
\newtheorem{remark}[theorem]{Remark}
\newtheorem{example}[theorem]{Example}

\numberwithin{equation}{section}

\def\1{\hbox{1\kern-.35em\hbox{1}}}

\newcommand{\N}{{\mathbb N}}
\newcommand{\Z}{{\mathbb Z}}

\newcommand{\C}{{\mathbb C}}

\newcommand{\tto}{\twoheadrightarrow}

\newcommand{\Soc}{{\rm Soc}}
\newcommand{\Rad}{{\rm Rad}}

\newcommand{\fa}{{\mathfrak a}}
\newcommand{\fb}{{\mathfrak b}}
\newcommand{\fc}{{\mathfrak c}}
\newcommand{\fg}{{\mathfrak g}}
\newcommand{\fh}{{\mathfrak h}}

\newcommand{\fl}{{\mathfrak l}}
\newcommand{\fn}{{\mathfrak n}}
\newcommand{\ofn}{\overline{\mathfrak n}}
\newcommand{\fp}{{\mathfrak p}}

\newcommand{\fu}{{\mathfrak u}}
\newcommand{\ofu}{\overline{{\mathfrak u}}}

\newcommand{\Hom}{{\rm Hom}}
\newcommand{\ind}{{\rm Ind}}
\newcommand{\coind}{{\rm Coind}}
\newcommand{\Ind}{{\rm Ind}^{\fg}_{\fg_{\oa}}}

\newcommand{\Ext}{{\rm Ext}}

\newcommand{\res}{{\rm Res}}
\newcommand{\ch}{{\rm ch}}
\newcommand{\Res}{{\rm Res}^{\fg}_{\fg_{\oa}}}

\newcommand{\U}{{\rm U}}
\newcommand{\Top}{{\rm Top}}

\newcommand{\cD}{{\mathcal D}}
\newcommand{\cL}{{\mathcal L}}
\newcommand{\cV}{{\mathcal V}}

\newcommand{\cC}{{\mathcal C}}
\newcommand{\cE}{{\mathcal E}}
\newcommand{\cF}{{\mathcal F}}

\newcommand{\cP}{{\mathcal P}}

\newcommand{\cO}{{\mathcal O}}

\newcommand{\cR}{{\mathcal R}}
\newcommand{\cU}{{\mathcal U}}

\newcommand{\oa}{\bar{0}}
\newcommand{\ob}{\bar{1}}

\begin{document}

\title[BBW theory for Lie superalgebras]{Bott-Borel-Weil theory and Bernstein-Gel'fand-Gel'fand reciprocity for Lie superalgebras}
\author{Kevin Coulembier}
\maketitle\vspace{-5mm}
{\center{\small{Department of Mathematical Analysis\\ Ghent University\\
Krijgslaan 281, Gent, Belgium}
}
\vspace{3mm}

\small{Department of Mathematics\\ University of California-Berkeley\\  
Evans Hall, Berkeley, USA

}}

\begin{abstract}
The main focus of this paper is Bott-Borel-Weil (BBW) theory for basic classical Lie superalgebras. We take a purely algebraic self-contained approach to the problem. A new element in this study is twisting functors, which we use in particular to prove that the top of the cohomology groups of BBW theory for generic weights is described by the recently introduced star action. We also study the algebra of regular functions, related to BBW theory. Then we introduce a weaker form of genericness, relative to the Borel subalgebra and show that the virtual BGG reciprocity of Gruson and Serganova becomes an actual reciprocity in the relatively generic region. We also obtain a complete solution of BBW theory for $\mathfrak{osp}(m|2)$, $D(2,1;\alpha)$, $F(4)$ and $G(3)$ with distinguished Borel subalgebra. Furthermore, we derive information about the category of finite dimensional $\mathfrak{osp}(m|2)$-modules, such as BGG-type resolutions and Kostant homology of Kac modules and the structure of projective modules.
\end{abstract}

\textbf{MSC 2010 :} 17B55, 18G10, 17B10\\
\noindent
\textbf{Keywords : }  Bott-Borel-Weil theory, algebra of regular functions, Zuckerman functor, $\mathfrak{osp}(m|2)$, projective module, generic weight, twisting functor, BGG reciprocity



\section{Introduction}

The first results on BBW theory for Lie supergroups were obtained by Penkov in \cite{MR0957752}. Up to now, only the case of basic classical Lie superalgebras of type I with distinguished Borel subalgebra is fully understood, see \cite{MR1680015, MR0957752, MR2059616}. The further study of BBW theory was mainly motivated by the quest for character formulae for finite dimensional representations of classical Lie superalgebras, see e.g. \cite{MR1309652, MR1036335}. Therefore, the character of the cohomology groups was of importance, rather than the $\fg$-module structure, and only BBW theory for dominant weights was relevant. The character problem was settled, with the aid of BBW theory, by Serganova in \cite{MR1443186} for $\mathfrak{gl}(m|n)$ and by Gruson and Serganova in \cite{MR2734963} for $\mathfrak{osp}(m|2n)$. BBW theory for the distinguished Borel subalgebra and {\em dominant} weights has been calculated for the algebras $\mathfrak{osp}(3|2)$ and $D(2,1;\alpha)$ by Germoni in \cite{Germoni} and for $G(3)$ and $F(4)$ by Martirosyan in \cite{Lilit}, all these are of type II. \\

In this paper, we are interested in the full $\fg$-module structure of the cohomology groups of BBW theory for arbitrary weights and Borel subalgebras. We also study the Zuckerman functor, the algebra of regular functions and twisting functors, which are ingredients for several BBW-type theories, including actual BBW theory. We take a purely algebraic and categorical approach to BBW theory, rather than the geometric approach in \cite{MR0957752}, but show the equivalence of both. This approach is closely related to the one of Santos in \cite{MR1680015} or Zhang in \cite{MR2059616}. This leads to a more direct derivation of results of algebraic nature in e.g. \cite{MR1680015, MR2734963, Gruson2} and a unifying treatment of main results on BBW theory for basic classical Lie superalgebras in \cite{MR1680015, MR2734963, Gruson2, MR0957752, MR1036335, MR2059616}.

One of the main {\em new} conclusions on BBW theory is that, for non-dominant weights, the cohomology groups are in general {\em{only}} highest weight modules {\em{if}} $\fg$ is of type I and {\em{if}} the distinguished Borel subalgebra is considered. This follows in particular from our restriction to the generic region, i.e. weights far away from the walls of the Weyl chamber. It is known from \cite{MR1036335} that, in that region, the cohomology is contained in one degree. We show that, even though the character of the corresponding module is given by the character of a highest weight module, the top of the module does not correspond to the simple subquotient with highest weight. Using twisting functors, introduced in \cite{MR2032059, MR2074588} and generalised to Lie superalgebras in \cite{MR3117249, CouMaz}, we prove that, while the character of the module is described by the dot action of the Weyl group, the top of the module is described by the star action. The star action is a deformation of the dot action, introduced in \cite{CouMaz}. In the generic region, the star action leads to an action of the Weyl group, which describes e.g. the primitive spectrum, see \cite{CouMaz}. Only for algebras of type I with distinguished Borel subalgebra, does the star action coincide with the dot action, in the generic region. 

We also obtain a full solution of BBW theory for $\mathfrak{osp}(m|2)$ for arbitrary $m$, $D(2,1;\alpha)$, $F(4)$ and $G(3)$ with distinguished Borel subalgebra, but for arbitrary weights. This confirms in particular the general results in the generic region.

Furthermore, we obtain several other homological results on the category of finite dimensional representations for $\mathfrak{osp}(m|2)$, relying on results of Su and Zhang in \cite{SuZha}. We calculate Kostant cohomology for Kac modules and discuss the existence of BGG type resolutions for these modules, revealing important differences with basic classical Lie superalgebras of type I.\\ 

For basic classical Lie superalgebras of type I with distinguished Borel subalgebra, the category of finite dimensional weight modules has the structure of a highest weight category, where the Kac modules are the standard modules. This resembles a parabolic category $\cO$ and in particular the BGG reciprocity holds, see e.g. \cite{MR2100468, MR1378540}. This was used by Brundan in \cite{Brundan} to provide an alternative solution to the character problem for $\mathfrak{gl}(m|n)$. For basic classical Lie superalgebras of type II, or those of type I regarded from the point of view of another Borel subalgebra, there is no analogue of the standard module and the category of finite dimensional modules is not of highest weight type.

The Kac module for basic classical Lie superalgebras of type I can be identified with the zero cohomology of BBW theory for integral dominant weights and the distinguished Borel subalgebra, and the higher cohomology groups are trivial. For arbitrary basic classical Lie superalgebras and Borel subalgebras, Gruson and Serganova associated a virtual module in the Grothendieck group to the cohomology groups of BBW theory for integral dominant weights in \cite{Gruson2}. This setup was used to prove a virtual BGG reciprocity. This was applied to find a solution for the character problem for $\mathfrak{osp}(m|2n)$, alternative to \cite{MR2734963}, and closer to the approach for $\mathfrak{gl}(m|n)$ in \cite{Brundan}.

In the current paper, we introduce a weaker version of the concept of generic weights, depending on the Borel subalgebra and called {\em relative genericness}. For the particular case of basic classical Lie superalgebras of type I with distinguished Borel subalgebra, the condition becomes trivial. We prove that for relatively generic weights, the cohomology groups of BBW theory are contained in one degree. This connects the corresponding result for type I with the one for generic weights. Then we show that in the relatively generic region, the zero cohomology of BBW theory for integral dominant weights, called generalised Kac modules, behave as standard modules. So, in particular, projective modules have a filtration by the standard modules, satisfying a BGG reciprocity relation, which strengthens the virtual BGG reciprocity of \cite{Gruson2} to a real one in the relatively generic region. For algebras of type I with distinguished Borel subalgebra, this recovers the full BGG reciprocity in \cite{MR1378540}. Also for algebras of type II with distinguished Borel subalgebra, the condition of relative genericness is very weak.\\

The paper is organised as follows. In Section \ref{secprel} we recall some preliminary notions. In Section \ref{secBBW} we obtain categorical reformulations of BBW theory in terms of the Zuckerman functor and Lie superalgebra cohomology. In Section \ref{algfun} we study the algebra of matrix elements of finite dimensional Lie supergroup representations. This is motivated by the role in BBW theory and the appearance of this algebra in physical theories, see e.g. \cite{MR2813506}. We also describe certain analogues of BBW theory. In Section \ref{secPenkov} we study the cohomology groups of BBW theory, restricted as $\fg_{\oa}$-modules. One particular motivation to do this originates from the subsequent results on generic weights and the special cases of $\mathfrak{osp}(m|2)$, $D(2,1;\alpha)$, $F(4)$ and $G(3)$. These results imply that, often, the modules appearing in higher cohomologies of BBW theory are the same ones as in the zero degree, when restricted to $\fg_{\oa}$-modules, but not as $\fg$-modules. In particular, we obtain an explicit finite complex which computes Kostant cohomology of projective modules in Theorem \ref{infocohomC}. In Section \ref{specseq} we briefly review the super analogues of the technique of Demazure in \cite{MR0229257} and show how it leads to a solution for BBW theory for typical weights and for basic classical Lie superalgebras of type I with distinguished Borel subalgebra. The results on BBW theory for generic weights are discussed in Section \ref{secgen}. In Section \ref{secrelgen} we define a version of genericness related to a particular parabolic subalgebra and show its relevance for BBW theory. In Section \ref{secBGGrec} we prove the generalised notion for BGG reciprocity in the relative generic region of the categories of finite dimensional representations. In Section \ref{seccomplres} we obtain the solution to BBW theory for $\mathfrak{osp}(m|2)$, $D(2,1;\alpha)$, $F(4)$ and $G(3)$ with distinguished Borel subalgebra. In Section \ref{sec32} we study homological properties of Kac modules for $\mathfrak{osp}(m|2)$. In Section \ref{HknP} we present a unifying formula for Kostant cohomology of projective modules for typical and generic weights. This formula also holds for arbitrary weights for the distinguished Borel subalgebra for either Lie superalgebras of type II with defect one, or Lie superalgebras of type I. We prove that this formula does not hold in general for Lie superalgebras of type II with defect higher than one. In Appendix \ref{aptwist} we derive some results on the twisting functors which are applied in other parts of the paper.

\section{Preliminaries}
\label{secprel}

\subsection{Basic classical Lie superalgebras}
\label{secprel1}
\begin{definition}
\label{catCab}
For any Lie superalgebra $\fc$, let $\cC(\fc)$ denote the category of $\fc$-modules. For a Lie algebra $\fa$, which is a subalgebra of $\fc$, let $\cC(\fc,\fa)$ denote the category of all $\fc$-modules which are locally $\U(\fa)$-finite and $\fa$-semisimple.
\end{definition}

The category $\cC(\fc,\fa)$ is sometimes referred to as the category of Harish-Chandra modules and denoted by $\mathcal{HC}(\fc,\fa)$.

We will always use the notation $\fg$ for a basic classical Lie superalgebra, see \cite{MR3012224, MR0519631, MR2906817}. The underlying Lie algebra is denoted by $\fg_{\oa}$ and the odd part by $\fg_{\ob}$, $\fg=\fg_{\overline{0}}\oplus\fg_{\ob}$.

There exist two types of basic classical Lie superalgebras (excluding Lie algebras), see Chapters 2 and 4 in \cite{MR2906817} for the explicit definition of the Lie superalgebras we introduce. For type I, the adjoint representation of $\fg_{\oa}$ in $\fg_{\ob}$ decomposes into two irreducible representations. Such Lie superalgebras have a $\Z$-grading of the form
\[\fg=\fg_{-1}\oplus\fg_0\oplus\fg_1\,,\quad\mbox{with}\quad \fg_{\oa}=\fg_0\quad\mbox{and}\quad \fg_{\ob}=\fg_{-1}\oplus\fg_1.\]
The list of basic classical Lie superalgebras of type I consists of $\mathfrak{osp}(2|2n)$, $\mathfrak{sl}(m|n)$ for $m\not=n$ and $\mathfrak{psl}(n|n)$.

For the basic classical Lie superalgebras of type II, the adjoint representation of $\fg_{\oa}$ in $\fg_{\ob}$ is irreducible. Such Lie superalgebras have a $\Z$-grading of the form
\[\fg=\fg_{-2}\oplus\fg_{-1}\oplus\fg_0\oplus\fg_1\oplus\fg_2\,,\quad\mbox{with}\quad \fg_{\oa}=\fg_{-2}\oplus\fg_0\oplus\fg_2\quad\mbox{and}\quad \fg_{\ob}=\fg_{-1}\oplus\fg_1.\]
The list of basic classical Lie superalgebras of type II consists of $\mathfrak{osp}(m|2n)$ for $m\not=2$, $D(2,1;\alpha)$, $F(4)$ and $G(3)$.

An arbitrary Borel subalgebra of $\fg$ (see Chapter 3 of \cite{MR2906817}) will be denoted by~$\fb$. Since these Borel subalgebras are not always conjugate under the action of the Weyl group, the BBW problem is not equivalent for different Borel subalgebras, so we can not restrict to one choice. However, we can consider the even part of the Borel subalgebra $(\fb_{\oa}=\fb\cap\fg_{\oa})$ to be fixed throughout the paper, without loss of generality. Tthe distinguished system Borel subalgebra (see \cite{MR0519631}) is denoted by $\fb^d$. The set of positive roots corresponding to the Borel subalgebra $\fb$ is denoted by $\Delta^+\subset\fh^\ast$. The relations $\Delta^+=\Delta^+_{\oa}\cup\Delta_{\ob}^+$ and $\Delta^+_{\oa}\cap\Delta_{\ob}^+=0$ hold, with $\Delta^+_{\oa}$ the set of even positive roots and $\Delta_{\ob}^+$ the set of odd positive roots. We define $\rho_{\oa}=\frac{1}{2}\sum_{\alpha\in\Delta^+_{\oa}}\alpha$, $\rho_{\ob}=\frac{1}{2}\sum_{\gamma\in\Delta^+_{\ob}}\gamma$ and $\rho=\rho_{\oa}-\rho_{\ob}$. For any non-isotropic root $\alpha$, we introduce $\alpha^\vee=2\alpha/\langle\alpha,\alpha\rangle$. If $\alpha$ is simple in $\Delta^+$, we have $\langle \alpha^\vee,\rho\rangle=1$.

The set of integral weights is denoted by $\cP\subset\fh^\ast$. The set of $\fg$-integral dominant weights $\cP^+\subset\fh^\ast$ is the set of weights such that there is a corresponding finite dimensional highest weight representation, this set depends on the choice of Borel algebra. Similarly, $\cP^+_{\oa}\subset\fh^\ast$ denotes the set of $\fg_{\oa}$-integrable dominant weights. Only for $\mathfrak{osp}(1|2n)$ and basic classical Lie superalgebras of type I with distinguished system of positive roots, we have the equality $\cP^+=\cP^+_{\oa}$. Otherwise, $\cP^+$ is a non-trivial subset of $\cP^+_{\overline{0}}$.

We denote by $\fp$ a parabolic subalgebra of the Lie superalgebra $\fg$, i.e. a subalgebra containing a Borel subalgebra $\fb$, see e.g. \cite{MR3012224, MR2906817}. We will always assume that the corresponding Levi subalgebra $\fl$ has the property that all its finite dimensional modules, which are semisimple for the Cartan subalgebra, are semisimple for the full algebra, then we say that $\fl$ is of {\em typical type}. This implies that $\fl$ is isomorphic to the direct sum of reductive Lie algebras and Lie superalgebras of the form $\mathfrak{osp}(1|2n)$. Unless $\fg\in\{\mathfrak{osp}(2d+1|2n), G(3)\}$, this condition is equivalent to $\fl\subset\fg_{\oa}$. The nilradical of the parabolic subalgebra $\fp$ is denoted by $\fu$, $\fp=\fl\oplus\fu$. The dual of the nilradical is denoted by $\overline{\fu}$, so $\fg=\overline{\fu}\oplus\fp$. The symbol $\fh$ is used for the Cartan subalgebra of $\fg$ contained in $\fb$, which is also a Cartan subalgebra of $\fg_{\oa}$. In case $\fp=\fb$, we have $\fl=\fh$ and then we use the notation $\fn$ for $\fu$. The dual of the Borel subalgebra is denoted by $\overline{\fb}=\fh\oplus\overline{\fn}$.

For any Lie superalgebra $\fc$ and a representation on a super vector space $M$, the dual representation is defined on the vector space $M^\ast=\Hom_\C(M,\C)$, as $(X\alpha)(v)=-(-1)^{|A||v|}\alpha(Xv)$ for $\alpha\in M^\ast$, $v\in M$ and $X\in \fc$. This module is also denoted by $M^\ast$. For $\fg$ a basic classical Lie superalgebra and $\dagger$ a Lie superalgebra automorphism mapping $\fg_{\alpha}$ to $\fg_{-\alpha}$, the twisted dual of a finite dimensional module $M$ is also described on the space $\Hom_\C(M,\C)$, but as $(X\alpha)(v)=\alpha(X^\dagger v)$. This module is denoted by $M^\vee$.

For $\fc$ a (parabolic subalgebra of a) basic classical Lie superalgebra or reductive Lie algebra, we denote the irreducible highest weight representation with highest weight $\mu$ by $L_{\mu}(\fc)$. We use the short-hand notation $L_\lambda=L_\lambda(\fg)$, $L_\lambda^{\oa}=L_\lambda(\fg_{\oa})$ and $L_\lambda^0=L_\lambda(\fg_0)$. We denote the Verma module for any $\mu\in\fh^\ast$ by $M_\mu=\U(\fg)\otimes_{\U(\fb)}L_\mu(\fb)$. If we want to mention the Borel subalgebra which is used in the definition of the Verma module or the simple module we use the notation $M^{(\fb)}_\mu$ and $L^{(\fb)}_\lambda$. 

We denote the indecomposable projective cover of $L_\lambda$ in the BGG category $\cO$ (see \cite{MR2100468, MR2906817}) by $P^\cO_\lambda $. For $\Lambda$ integral dominant we denote the indecomposable projective cover of $L_\Lambda$ in the category $\cF$ of finite dimensional weight modules (see \cite{Gruson2, MR1443186, MR1378540}) by $P^\cF_\Lambda $. Category $\cO$ is naturally isomorphic to a subcategory of $\cC(\fg,\fh)$ and will sometimes silently be identified with this subcategory. To make a distinction we denote the BGG category for $\fg_{\oa}$ by $\cO_{\oa}$.

\subsection{Actions of the Weyl group}
\label{secprelWeyl}
For Basic classical Lie superalgebras, the Weyl group $W=W(\fg:\fh)$ corresponds to the Weyl group $W(\fg_{\oa}:\fh)$. For any $\alpha$, simple in $\Delta_{\oa}^+$, we denote the simple reflection by $s_\alpha$. The length of an element $w$ of the Weyl group is denoted by $l(w)$. The set of all elements with length $p$ is denoted by $W(p)$.

For each system of positive roots, the $\rho$-shifted action of the Weyl group $W$ of $\fg$ is denoted by $w\cdot\lambda=w(\lambda+\rho)-\rho$ for $w\in W$ and $\lambda\in\fh^\ast$. As in \cite{MR2906817}, we will denote the $\rho_{\oa}$-shifted action of $W$ on $\fh^\ast$ by $w\circ\lambda=w(\lambda+\rho_{\oa})-\rho_{\oa}$.

We will need the following two sets:
\begin{equation}
\label{gammaset}
\Gamma^+=\{\sum_{\alpha\in I}\alpha\,|I\,\subset \Delta^+_{\ob}\}\quad\mbox{and}\quad \widetilde{\Gamma}=\{\sum_{\alpha\in I}\alpha\,|\,I\subset \Delta_{\ob}\}.
\end{equation}
Note that we interpret these sets with {\em multiplicities}, so even if $\sum_{\alpha\in I}\alpha=\sum_{\alpha\in I'}\alpha$, the left-hand and right-hand are regarded as two distinct elements if $I\not= I'$. We have the equality of sets \begin{equation}\label{MusGS}w\circ(\lambda-{\Gamma}^+)=w\cdot \lambda -{\Gamma^+} \quad \mbox{for any }\lambda\in\fh^\ast,\end{equation}
see Section 0.5 in \cite{MR1479886} or the proof of Lemma 3 in \cite{MR2734963}.

At certain points we will derive results specific to weights far away from the walls of the Weyl chambers. Such weights are often called generic and the corresponding highest weight modules have been studied in e.g. \cite{CouMaz, MR1309652, MR1036335}. In the following, the notion of Weyl chambers refers to the Weyl chambers of the $\rho_{\oa}$-shifted action.

\begin{definition}
\label{defgeneric}
(i) A weight $\lambda\in\fh^\ast$ is {\em $\Gamma^+$-generic} if all weights in the set $\lambda-\Gamma^+$ are inside the same Weyl chamber.

(ii) A weight $\lambda\in\fh^\ast$ is {\em $\widetilde{\Gamma}$-generic} if all weights in the set $\lambda-\widetilde{\Gamma}$ are inside the same Weyl chamber.

(iii) A weight $\lambda\in\fh^\ast$ is called {\em generic} if every weight in the set $\lambda-\Gamma^+$ is $\widetilde{\Gamma}$-generic. 
\end{definition}
The set $\widetilde{\Gamma}$ is invariant under the Weyl group, which is a consequence of the fact that $\Lambda \fg_{\ob}$ is a finite dimensional $\fg_{\oa}$-module. Thus, a weight $\lambda$ is $\widetilde{\Gamma}$-generic if and only if $w\circ\lambda$ is $\widetilde{\Gamma}$-generic for an arbitrary $w\in W$. Furthermore, equation \eqref{MusGS} implies that a weight $\lambda$ is $\Gamma^+$-generic if and only if $w\cdot\lambda$ is $\Gamma^+$-generic for an arbitrary $w\in W$. By the same reason, $\lambda$ is generic if and only if $w\cdot\lambda$ is generic for $w\in W$.

We note that the notion of $\widetilde{\Gamma}$-generic weight in Definition \ref{defgeneric} is identical to the notion of weakly generic weights of Definition 7.1 in \cite{CouMaz}. The notion of genericness of Definition \ref{defgeneric} coincides with the one in Definition 7.1 in \cite{CouMaz}.

Since we assume that two different Borel subalgebras have the same underlying even Borel subalgebra $\fb_{\oa}=\fb\cap\fg_{\oa}$, the notion of a highest weight module does not depend on the choice of $\fb$. Subsequently the BGG category $\cO$ coincides for both Borel subalgebras, even though the structure as a highest weight category differs. How the highest weight of a highest weight representations in different systems of positive roots are related is described by the technique of odd reflections, see e.g. \cite{MR2906817, MR2743764}.

The $\rho$-shifted action of the Weyl group depends {\em essentially} on the choice of Borel subalgebra in the atypical region. More precisely, the {\em sets of} simple {\em modules}, of highest weight type, linked together by the condition that the highest weights, for the system of positive roots $\Delta^+$, are in the same $\rho$-shifted orbit, are different for each choice of $\Delta^+$. This is possible since atypical central characters correspond to an infinite amount of Weyl group orbits (for a fixed Borel subalgebra), contrary to the situation for simple Lie algebras. For $\widetilde{\Gamma}$-generic weights this can be solved by considering star actions as in Section 8.1 in \cite{CouMaz}, as explained underneath. If we want to mention the Borel subalgebra which is used explicitly, we denote the star action by $\ast^{\fb}$. The principle of this action can be described as follows. For a weight $\lambda$ and a simple reflection $s_{\alpha}$, denote $\widetilde{\lambda}$ the highest weight of $L^{(\fb)}_\lambda$ in a system of positive roots $\widetilde{\Delta}^+$ (with $\widetilde{\Delta}^+_{\oa}=\Delta^+_{\oa}$ and with corresponding Borel subalgebra $\widetilde{\fb}$) in which $\alpha$ or $\alpha/2$ is simple, so $L^{(\fb)}_\lambda\cong L^{(\widetilde{\fb})}_{\widetilde{\lambda}}$. The simple star reflection $s_\alpha\ast^{\fb} \lambda$ is then defined as the highest weight of the module $L^{(\widetilde{\fb})}(s_{\alpha}(\widetilde{\lambda}+\widetilde{\rho})-\widetilde{\rho})$ in the system of positive roots $\Delta^+$. The results of \cite{CouMaz} then imply that for $\widetilde{\Gamma}$-generic weights this (i) leads to an action of the Weyl group and (ii) is independent of the choice of the specific $\widetilde{\Delta}^+$. By definition, we then have that for a $\widetilde{\Gamma}$-generic $\Lambda\in\cP^+$ and two Borel subalgebras $\fb$ and $\hat{\fb}$ with $\fb_{\oa}=\widetilde{\fb}_{\oa}$
$$L^{(\fb)}_{w\ast^{\fb}\Lambda}\,\,\cong\,\, L^{(\widetilde{\fb})}_{w\ast^{\tilde{\fb}}\widetilde{\Lambda}}\quad\mbox{for every}\quad w\in W,$$
with $\widetilde{\Lambda}\in\fh^\ast$ defined by the relation $L^{(\fb)}_\Lambda\cong L^{(\widetilde\fb)}_{\widetilde\Lambda}$. Therefore, in the generic region the star action of the Weyl group does not depend {\em essentially} on the choice of Borel subalgebra.

Only in case $\fg$ is of type I and $\fb=\fb^d$, we have the equality $w\ast^{\fb^d}\lambda=w(\lambda+\rho^d)-\rho^d$.

Consider a Levi subalgebra $\fl_{\oa}$ of $\fg_{\oa}$. Every $w\in W$ decomposes as $w=w_1 w^1$ with $w_1\in W(\fl_{\oa}:\fh)$ and where $w^1$ maps $\fl_{\oa}$-dominant weights to $\fl_{\oa}$-dominant weights, see Propositions 3.4 and 3.5 in \cite{lepowsky}. We denote the set of all such $w^1$ by $W^1(\fl_{\oa})$.

\subsection{Zuckerman functor, induction functor and generalised Kac modules}
\label{subsecprelZuck}

Now we introduce the Zuckerman functor, see \cite{MR0894570, MR0563362, MR2331754, MR2059616}.

\begin{definition}
\label{defZuck}
Consider a Lie superalgebra $\fl$ of typical type, which is a subalgebra of a basic classical Lie superalgebra $\fg$. The Zuckerman functor
\[S:\cC(\fg,\fl)\to\cC(\fg,\fg_{\overline{0}})\]
sends a module $M$ in the category $\cC(\fg,\fl)$ to $M[\fg_{\oa}]$, the maximal $\fg$-submodule which is locally $\U(\fg_{\oa})$-finite and $\fg_{\oa}$-semisimple.
\end{definition}
This is a left exact functor (see e.g. Lemma 4.1 in \cite{MR2059616}) and its right derived functors are denoted by $\cR_k S:\cC(\fg,\fl)\to\cC(\fg,\fg_{\overline{0}})$.

We define the Bernstein functor as the adjoint of the Zuckerman functor. 
\begin{definition}
The Bernstein functor $\Gamma:\cC(\fg,\fl)\to\cC(\fg,\fg_{\oa})$ maps $M$ to its maximal locally finite quotient of a module $M\in\cC(\fg,\fl)$. This functor is right exact and its left derived functors are denoted by $\cL_k\Gamma$.
\end{definition}
From definition, it follows that if $M\in\cO$, then $\cR_kS(M)=\left(\cL_k\Gamma(M^\vee)\right)^\vee$, with $\vee$ the duality on $\cO$, see Section 3.2 in \cite{MR2428237}.

Consider two Lie superalgebras $\fc_1$ and $\fc_2$ such that $\fc_2$ is a subalgebra of $\fc_1$. We denote the forgetful functor $\cC(\fc_1)\to\cC(\fc_2)$ by $\res^{\fc_1}_{\fc_2}$. The same notation will be used for the forgetful functor $\cC(\fc_1,\fa_1)\to\cC(\fc_2,\fa_2)$ if also $\fa_2\subset\fa_1$. The induction and coinduction functors are denoted respectively by $\ind^{\fc_1}_{\fc_2}:\cC(\fc_2)\to\cC(\fc_1)$ and $\coind^{\fc_1}_{\fc_2}:\cC(\fc_2)\to\cC(\fc_1)$. Their action on a $\fc_2$-module $V$ is given by
$$\ind^{\fc_1}_{\fc_2}V=\U(\fc_1)\otimes_{U(\fc_2)}V\qquad\mbox{and}\qquad \coind^{\fc_1}_{\fc_2}V=\Hom_{\U(\fc_2)}(\U(\fc_1),V).$$

We summarise a few facts about these functors, which will be useful in later sections.
\begin{lemma}
\label{lemind}
(i)For any basic classical Lie superalgebra $\fg$ with parabolic subalgebra $\fp=\fl\oplus\fu$ such that $\fl$ is of typical type, the functor $\coind^{\fg}_{\fp}$ restricts to a functor
$$\coind^{\fg}_{\fp}:\cC(\fp,\fl)\to\cC(\fg,\fl),$$
which is exact. Moreover, this functor maps injective modules in $\cC(\fp,\fl)$ to injective modules in $\cC(\fg,\fl)$.

(ii) The functors $\Ind$ and $\coind^{\fg}_{\fg_{\oa}}$ are isomorphic.
\end{lemma}
\begin{proof}
Consider $V\in\cC(\fp,\fl)$. As an $\fl$-module we have
$$\Hom_{\U(\fp)}(\U(\fg),V)\cong \left(\U(\overline{\fu})\right)^\ast\otimes V.$$
Since the $\fl$-module $\U(\overline{\fu})\cong S(\overline{\fu})=\bigoplus_{k=0}^\infty S^k(\overline{\fu})$ is the direct sum of finite dimensional $\fl$-modules, it follows that $\Hom_{\U(\fp)}(\U(\fg),V)\in\cC(\fg,\fl)$. Its exactness is proved in Section 2 of \cite{MR2100468}. The fact that it maps injective modules to injective modules is proved in Corollary 4.1 of \cite{MR2059616}. This concludes $(i)$.

Part $(ii)$ follows from the fact that $U(\fg_{\oa})\hookrightarrow U(\fg)$ is a finite ring extension and the fact that the $\fg_{\oa}$-module $\Lambda\fg_{\ob}$ is self-dual. \end{proof}

Contrary to the classical case, the maximal finite dimensional quotient of an integral dominant Verma module is not the corresponding simple finite dimensional module. We introduce the following notation for the corresponding module,
$$K^{(\fb)}_\Lambda:=\Gamma(M^{(\fb)}_\Lambda).$$
So $K^{(\fb)}_{\Lambda}$ is the maximal finite dimensional highest weight module wight highest weight $\Lambda$ and we have $\Top(K^{(\fb)}_\Lambda )\cong L^{(\fb)}_\Lambda$.

For each basic classical Lie superalgebra $\fg$ and $\Lambda\in\cP^+$ we define the Kac module $K_\Lambda$ as in \cite{MR0519631}. For $\fg$ of type I this module is defined as the parabolically induced module
\[K_\Lambda=\U(\fg)\otimes_{\U(\fg_0\oplus \fg_1)}L_{\Lambda}^0.\]
For $\fg$ of type II the Kac module is defined as
\begin{equation}
\label{KacII}
K_{\Lambda}=\overline{K}_\Lambda/N_{\Lambda}\qquad\mbox{with}\quad \overline{K}_\Lambda=\U(\fg)\otimes_{\U(\fg_0\oplus \fg_1\oplus \fg_2)}L_\Lambda^0\end{equation}
and $N_\Lambda=\U(\fg)Y_{-\phi}^{b+1}\otimes L_\Lambda^0$ with $\phi$ the longest simple positive root of $\fg_{\oa}$ hidden behind the odd simple root and $b=\langle \Lambda,\phi^\vee\rangle$. For both cases we have $K_\Lambda=K_\Lambda^{(\fb^d)}$ with $\fb^d$ the distinguished Borel subalgebra, see Lemma \ref{zerocohomKac}. Because of this property we will call the module $K_{\Lambda}^{(\fb)}$, for an arbitrary Borel subalgebra $\fb$, a generalised Kac module.

\subsection{Lie superalgebra cohomology and twisting functors}
\label{prelsec4}
We will make extensive use of the algebra (co)homology of the nilradical $\fu$ of the parabolic subalgebra $\fp$. We denote by $H^k(\fu,M)$ the $k$-th cohomology group of $\fu$-cohomology with values in the $\fu$-module $M$ and by $H_k(\fu,M)$ the $k$-th homology group. When $M$ is considered to be a (finite dimensional or unitarisable) $\fg$-module, this is usually referred to as Kostant cohomology and was studied in the Lie algebra setting in \cite{MR0142696}. For Lie superalgebras, an overview of the definitions, some basic properties and connection with Ext functors is presented in Section 6.4 in \cite{MR3012224}, Section 4 of \cite{BGG} or Chapter 16 in \cite{MR2906817}. If $M$ is a $\fg$-module, the $\fu$-(co)homology groups are naturally $\fl$-modules.

For $V\in \cC(\fp,\fl)$, and $\mu\in\fh^\ast$ an integral dominant $\fl$-weight, the equality
\begin{equation}
\label{ExtHk}
\Hom_{\fl}(L_{\mu}(\fl),H^k(\fu,V))\,=\,\Ext^k_{\cC(\fp,\fl)}(L_\mu(\fp),V)
\end{equation}
follows from the equalities $\Hom_{\fl}(L_{\mu}(\fl),H^k(\fu,V))=\Hom_{\fl}(\C,H^k(\fu,L_{\mu}(\fl)^\ast\otimes V))$ and $\Ext^k_{\cC(\fp,\fl)}(L_\mu(\fp),V)=\Ext^k_{\cC(\fp,\fl)}(\C,L_\mu(\fp)^\ast\otimes V)$ and the fact that the standard projective resolution of $\C$ in $\cC(\fu)$, see e.g. Section 6.5.2 in \cite{MR3012224}, Lemma 4.7 in \cite{BGG} or Section 7 in \cite{MR1269324}, can be interpreted as a projective resolution in $\cC(\fp,\fl)$.

In Section 5 of \cite{CouMaz} the twisting functor $T_\alpha$ on category $\cO$ was introduced for every $\alpha$ simple in $\Delta^+_{\oa}$. This is a generalisation of the Arkhipov twisting functor on category $\cO$ for semisimple Lie algebras, studied in e.g. \cite{MR2032059, MR2074588, MR2366357, MR2331754}. The twisting functors are right exact and we denote the left derived functors by $\cL_i T_\alpha$. If we denote by $T_{\alpha}^{\oa}$ the twisting functor on $\cO_{\oa}$, Lemma 5.1 and equation (5.1) in \cite{CouMaz} states the following useful properties:
\begin{equation}
\label{eqlemma51}
\cL_iT_\alpha \circ \Ind\cong \Ind\circ \cL_iT^{\oa}_\alpha\quad\mbox{and}\quad \Res\circ \cL_iT_\alpha\cong \cL_iT_\alpha^{\oa}\circ \Res\quad \mbox{for }i\in\N.
\end{equation}
The twisting functors satisfy braid relations, so in particular we can define the functor $T_w$ for $w\in W$ as the the composition $T_{\alpha_1}\circ T_{\alpha_2}\cdots T_{\alpha_p}$ for $s_{\alpha_1}s_{\alpha_2}\cdots s_{\alpha_p}$ an arbitrary reduced expression for $w$, see Lemma 5.3 in \cite{CouMaz}. The right adjoint functor of $T_\alpha$ on $\cO$ is denoted by $G_\alpha$. By definition, this functor inherits the intertwining properties in equation \eqref{eqlemma51} and the braid relations from $T_\alpha$.

The twisting functors have an interesting relation with Verma modules.
\begin{lemma}[Lemma 5.7 in \cite{CouMaz}]
\label{TVerma}
Consider $\alpha$ simple in $\Delta_{\oa}^+$ and $\lambda\in\fh^\ast$. Assume that either\begin{itemize}
\item $\alpha$ or $\alpha/2$ is simple in $\Delta^+$, or 
\item $\lambda$ is typical.
\end{itemize}
Then $T_\alpha M_\lambda=M_{s_\alpha\cdot \lambda}$ unless $\langle \lambda,\alpha^\vee\rangle\in\Z$ with $\langle \lambda+\rho,\alpha^\vee\rangle<0$.
\end{lemma}
In the current paper we will derive some further properties of these twisting functors in Appendix \ref{aptwist}.


\section{Reformulations of Bott-Borel-Weil theory}
\label{secBBW}

We use the notation $\fg$ for a basic classical Lie superalgebra with parabolic subalgebra $\fp=\fl\oplus\fu$ as in the prelimiaries (so in particular $\fl$ is of typical type). BBW theory is defined through a connected Lie supergroup $G$ (with Lie superalgebra $\fg$) with a subsupergroup $P$ with Lie superalgebra $\fp$, see e.g. \cite{MR2734963, Gruson2, MR0957752, MR1036335}. Consider a $P$-module $V$ and the corresponding vector bundle $\cV=G\times_P V$. In \cite{MR0957752} the sheaf cohomology or \v{C}ech cohomology on such vector bundles was introduced. Since the sheaf of sections on $\cV$ is a $\fg$-sheaf, the space of holomorphic sections $H^0(G/P,\cV)$ and the higher cohomology groups $H^k(G/P,\cV)$ are $\fg$-modules. As in \cite{MR2734963, Gruson2} we define $\Gamma_k(G/P,V)=H^k(G/P,G\times_P V^\ast)^\ast$. We are interested in calculating $\Gamma_k(G/P,L_{\mu}(\fp))$ for an $\fl$-dominant $\mu\in \cP$. The main results of this section are summarised in the following proposition and theorem.
\begin{proposition}
\label{propBBW}
The cohomology groups of the $\fg$-sheaf of sections on the vector bundle $G\times_P V$ with $V\in\cC(\fp,\fl)$ satisfy
\begin{eqnarray*}
(i)&& H^k(G/P,G\times_P V))=\cR_k S(\coind^{\fg}_{\fp}(V))\\
(ii)&& \Gamma_k(G/P,V)=\cL_k\Gamma(\ind^{\fg}_{\fp}V)\\
(iii)&& H^k(G/P,G\times_P V)=\Hom_{\U(\fl)}(\C, H^k(\fu,V\otimes\cR)),
\end{eqnarray*}
with $\cR\cong\C[G]$ the $\fg$-bimodule corresponding to the algebra of matrix elements of the finite dimensional weight modules of $\fg$ (finite dimensional $G$-modules).
\end{proposition}
This $\fg$-bimodule $\cR$ will be studied in full extend in Section \ref{algfun}.
\begin{theorem}
\label{thmBBW}
For any integral dominant $\fl$-weight $\mu\in\fh^\ast$, we have
\begin{eqnarray*}
(i)&& \Gamma_k(G/P,L_{\mu}(\fp))= \Gamma_k(G/B,L_{\mu}(\fb))\\
(ii)&& \Gamma_k(G/B,L_\mu(\fb))=\cL_k\Gamma(M(\mu))=\left(\Ext_{\cO}^k(M(\mu),\cR)\right)^\ast\\
(iii)&& H^k(G/B,G\times_B L_{-\mu}(\fb))=\Ext_{\cO}^k(M(\mu),\cR).
\end{eqnarray*}
\end{theorem}
In particular Theorem \ref{thmBBW}(i) implies that the solution of BBW theory for the Borel subalgebra is sufficient for our range of parabolic subalgebras. The remainder of this section is mainly devoted to proving Proposition \ref{propBBW} and Theorem \ref{thmBBW}. We note that Theorem \ref{thmBBW}(i) can also be obtained as a special case of the Leray spectral sequence in Theorem 1 of \cite{MR2734963}, but we provide an alternative proof.

\begin{proof}[Proof of Proposition \ref{propBBW}]
The $\fg$-module $H^0(G/P,G\times_P V)$ is isomorphic to 
\begin{equation}\label{RBBW}\left(\cR\otimes V\right)^P=S(\Hom_{\U(\fp)}(\U(\fg),V))=S(\coind^{\fg}_{\fp}(V))\end{equation} see e.g. the proof of Lemma 2 in \cite{MR2734963} and the subsequent Lemma \ref{findual2}. Since these identities are functorial we obtain that the functors $H^0(G/P,G\times_P -)$ and $S\circ \coind^{\fg}_{\fp}$ acting between
\[\cC(\fp,\fl)\to\cC(\fg,\fg_{\oa})\]
are identical. Therefore their derived functors are also identical. Since $\coind^{\fg}_{\fp}:\cC(\fp,\fl)\to\cC(\fg,\fl)$ is exact and maps injective modules to injective modules, see Corollary 4.1 in \cite{MR2059616}, the right derived functors of the left exact functor $S\circ \coind^{\fg}_{\fp}$ are given by $\cR_k(S\circ \coind^{\fg}_{\fp})= \cR_k S\circ \coind^{\fg}_{\fp}$.

Since the functors $H^k(G/P,G\times_P-)$ are the right derived functors of the left exact functor $H^0(G/P,G\times_P-):\cC(\fp,\fl)\to\cC(\fg,\fg_{\oa})$, Proposition \ref{propBBW}(i) follows.

Proposition \ref{propBBW}(ii) is just a reformulation of this result.

Proposition \ref{propBBW}(iii) can be proved similarly to Lemma 5.1 in \cite{MR0563362}, but here we take a more direct approach. For $k=0$, equation \eqref{RBBW} implies 
$$H^0(G/P,G\times_P V)\cong\Hom_{\fp}(\C,\cR\otimes V)\cong\Hom_{\cC(\fp,\fl)}(\C,\cR\otimes V).$$ The equality of the higher cohomologies then follows from taking derived functors and equation \eqref{ExtHk}.
\end{proof}

Before proving Theorem \ref{thmBBW} we focus on the following lemma.
\begin{lemma}
\label{KostantBott}
Let $\mu$ be an integral dominant $\fl$-weight and $V\in\cC(\fp,\fl)$, then
\[\Ext^k_{\cC(\fp,\fl)}(L_{\mu}(\fp),V)\quad\cong\quad\Ext^k_{\cC(\fb,\fh)}(L_{\mu}(\fb),V)\qquad\mbox{and}\]
\[L_\mu(\fl)\subset H^k(\fu,V)\qquad \Leftrightarrow\qquad \C_\mu\subset H^k(\fn,V).\]
\end{lemma}
\begin{proof}
We prove, more generally, that the functors
\begin{eqnarray}\label{ExtExt}\Ext^k_{\cC(\fp,\fl)}(L_{\mu}(\fp),-)\qquad\mbox{and}\qquad \Ext^k_{\cC(\fb,\fh)}(L_{\mu}(\fb),-)\circ \res^\fp_{\fb}, \end{eqnarray}
acting from $\cC(\fp,\fl)$ to \textbf{Set}, are isomorphic. This property clearly holds for $k=0$. 

Now we prove that injective modules in $\cC(\fp,\fl)$ are mapped by $\res^{\fp}_{\fb}$ to acyclic modules of the functor $\Hom_{\cC(\fb,\fl)}(L_{\mu}(\fb),-)$. These injective modules are direct summands of modules of the form $I=\Hom_{\U(\fl)}(\U(\fp),L_\kappa(\fl))$, see \cite{Hochschild}. Since we have
$$\res^{\fp}_{\fb}I\cong \Hom_{\U(\fb_{\fl})}(\U(\fb),\res^{\fl}_{\fb_{\fl}}L_\kappa(\fl)),$$
with $\fb_{\fl}:=\fb\cap \fl$, we can apply equation \eqref{ExtHk} and the Frobenius reciprocity to obtain
\begin{eqnarray*}
\Ext_{(\fb,\fh)}^k(L_{\mu}(\fb),I)&=& \Ext_{(\fb_{\fl},\fh)}^k(L_{\mu}(\fb_{\fl}),L_{\kappa}(\fl))\\
&=&\Hom_{\fh}(\C_\mu,H^k(\fn_{\fl}, L_{\kappa}(\fl))).
\end{eqnarray*}
Since $\mu$ is $\fl$-integral dominant and $\fl$ is of typical type, Kostant cohomology for $\fl=\fn_{\fl}^-\oplus\fh\oplus \fn_{\fl}$ implies that the expression above can only be non-zero if $k=0$, see \cite{MR0563362, MR0142696, osp12n}. This proves that $\res^{\fp}_{\fl}$ maps injective modules in $\cC(\fp,\fl)$ to acyclic modules for $\Hom_{\cC(\fb,\fl)}(L_{\mu}(\fb),-)$ if $\mu$ is $\fl$-integral dominant.

The Grothendieck spectral sequence of Section 5.8 in \cite{MR1269324} then implies that the functor on the right-hand side of equation \eqref{ExtExt} is the derived functor of the functor for $k=0$, from which the equality follows.
\end{proof}

\begin{proof}[Proof of Theorem \ref{thmBBW}]
Proposition \ref{propBBW}(iii) and equation \eqref{ExtHk} imply that
\begin{eqnarray}
\nonumber
H^k(G/P,G\times_P L_{\mu}(\fp)^\ast)&=&\Hom_{\U(\fl)}(L_{\mu}(\fp), H^k(\fu,\otimes\cR))\\
\label{thmBBWhelp}
&=& \Ext^k_{\cC(\fp,\fl)}(L_{\mu}(\fp),\cR).
\end{eqnarray}
Theorem \ref{thmBBW}(i) therefore follows from Lemma \ref{KostantBott}.

Equation \eqref{thmBBWhelp} implies, through the Frobenius reciprocity, that we have
$$H^k(G/P,G\times_B L_{\mu}(\fb)^\ast)= \Ext^k_{\cC(\fg,\fh)}(M_{\mu},\cR).$$
Since $\cO$ is extension full in $\cC(\fg,\fh)$, see Theorem 24 in \cite{CouMaz2} or Theorem 6.15 in \cite{MR2428237}, this implies Theorem \ref{thmBBW}(iii). The first equality in Theorem \ref{thmBBW}(ii) is a special case of Proposition \ref{propBBW}(ii), the second equality is an immediate reformulation of Theorem~\ref{thmBBW}(iii).
\end{proof}

\begin{corollary}
\label{corcentr}
The $\fg$-module $\Gamma_k(G/P,L_{\mu}(\fp))$ admits the central character $\chi_\mu$.
\end{corollary}

Now we show that the Kac modules for both types of basic classical Lie superalgebras are a special case of the generalised Kac modules and thus correspond to the zero cohomology of BBW theory for the distinguished Borel subalgebra.
\begin{lemma}
\label{zerocohomKac}
\label{lemKacquo}
Consider a basic classical Lie superalgebra $\fg$ with distinguished Borel subalgebra $\fb^d$ and $\Lambda\in\cP^+$ an integral dominant weight. The maximal finite dimensional quotient $K^{(\fb^d)}_\Lambda=\Gamma(M^{(\fb^d)}_{\Lambda})$ of the Verma module $M^{(\fb^d)}_\Lambda$ is equal to the Kac module $K_\Lambda$.
\end{lemma}
\begin{proof}
The lemma follows from the fact that all the $\fg_{\oa}$-highest weight vectors in $M_\Lambda$ which do not have an integral dominant highest weight need to be inside the submodule that will be factorised out, the fact that $K_\Lambda$ is finite dimensional and the definition of the Kac modules in Subsection \ref{subsecprelZuck}.
\end{proof}


\section{The algebra of regular functions and the Zuckerman functor}
\label{algfun}

\subsection{The algebra of regular functions}
\label{subsecalgreg}

In this subsection we study the $\fg$-bimodule $\cR$ that appeared in Proposition \ref{propBBW} and Theorem \ref{thmBBW}, given by the algebra of regular functions on the supergroup $G$. The universal enveloping algebra $\U(\fg)$ is a $\fg\times\fg$-module for left and right multiplication. The dual space $\U(\fg)^\ast=\Hom_\C(\U(\fg),\C)$ inherits a $\fg\times\fg$-representation structure from $\U(\fg)$. The universal enveloping algebra also possesses the structure of a super cocommutative Hopf superalgebra with comultiplication $\Delta(X)=X\otimes 1+1\otimes X$ for $X\in \fg$, see e.g. \cite{MR1243637, MR2059616}. This gives $\U(\fg)^\ast$ the structure of a super commutative Hopf superalgebra.

\begin{lemma}
\label{findual}
Taking the maximal locally finite submodule from the left or right of the $\fg$-bimodule $\U(\fg)^\ast$, gives the same $\fg\times\fg$ submodule $\U(\fg)^{\circ}$. This module is isomorphic to the finite dual of the Hopf superalgebra $\U(\fg)$. In particular, this gives $\cU(\fg)^0$ the structure of a super commutative algebra. 

As an algebra and as a $\fg$-bimodule, $\U(\fg)^\circ$ is isomorphic to the algebra of matrix elements of finite dimensional $\fg$-representations.
\end{lemma}
\begin{proof}
The first paragraph follows from Lemma 9.1.1 in \cite{MR1243637} or Lemma 3.1 in \cite{MR2059616}. 

A matrix element of a $\fg$-module $V$ is in particular an element of $\U(\fg)^\ast$. The left and right $\fg$-action acting on that matrix element generate a subquotient of $V^\ast\otimes V$, so in particular, if $V$ is finite dimensional, the matrix elements constitute a locally finite $\fg\times \fg$-submodule of $\U(\fg)^\ast$, so a submodule of $\U(\fg)^\circ$. Similarly, every element of $\U(\fg)^\circ$ can be interpreted as the matrix element of a finite dimensional representation.
\end{proof}

\begin{lemma}
\label{findual2}
The algebra $\cR$, of matrix elements of finite dimensional weight modules of $\fg$, is isomorphic to the module obtained by to taking the maximal $\fh$-semisimple submodule of $\U(\fg)^{0}$ on the left or right hand side. Consequently we have $\cR=S(U(\fg)^\ast),$ with the Zuckerman functor acting from the right or the left.
\end{lemma}
\begin{proof}
This follows from the interpretation of $\cR$ and $\U(\fg)^\circ$, respectively as matrix elements of finite dimensional modules and finite dimensional weight modules.
\end{proof}

Based on Proposition \ref{propBBW}(iii) and Theorem \ref{thmBBW}(ii) for $k=0$  and Lemma \ref{lemKacquo}, we obtain the following property.
\begin{corollary}
\label{zerocohomR}
The $\fg$-bimodule $\cR$ satisfies the property
\[H^0(\fu,\cR)\cong\bigoplus_{\Lambda\in\cP^+}L_{\Lambda}(\fl)\times (K^{(\fb)}_\Lambda)^\ast\qquad\mbox{as $\fl\times\fg$-modules.}\]
\end{corollary}
As in \cite{MR0237590}, we define the $\fh\times\fg$-module $F(\fg)=H^0(\fn,\cR)$ and the $\fg$-module
\[{}^\mu F(\fg)=\Hom_{\fh}(\C_\mu,F(\fg))=\Hom_{\fb}(\C_\mu,\cR),\]
for $\mu\in\fh^\ast$. Since the elements of $\fn$ act as super derivatives on $\cR$ (satisfying a graded Leibniz rule), the subspace $F(\fg)$ of $\cR$ is actually a subalgebra.

The following theorem extends the result of Wallach in Theorem 5.1 of \cite{MR0237590}.
\begin{theorem}
\label{genWallach}
The $\fg$-module $F(\fg)$ contains every module $K_\Lambda^{(\fb)}$ exactly once, $$F(\fg)\cong \bigoplus_{\Lambda\in\cP^+}\left(K_\Lambda^{(\fb)}\right)^\ast.$$ Furthermore, ${}^\mu F(\fg)\cong \left(K^{(\fb)}_\mu\right)^\ast$ if $\mu$ is integral dominant and ${}^\mu F(\fg)=0$ otherwise. Within the algebra structure of $F(\fg)\subset \cR$, the relation \[{}^\Lambda F(\fg)\cdot {}^{\Lambda'} F(\fg)={}^{\Lambda+\Lambda'}F(\fg)\]
holds for $\Lambda$ and $\Lambda'$ integral dominant.
\end{theorem}
\begin{proof}
The first two statements are immediate consequences of Corollary \ref{zerocohomR} for $\fu=\fn$. By definition, the property ${}^\Lambda F(\fg)\cdot {}^{\Lambda'} F(\fg)\subset {}^{\Lambda+\Lambda'}F(\fg)$ follows immediately. It remains to be proved that this product is surjective.

First, we prove the existence of an injective $\fg$-module morphism
\begin{equation}\label{injection}K^{(\fb)}_{\Lambda+\Lambda'}\hookrightarrow K_{\Lambda}^{(\fb)}\otimes K_{\Lambda'}^{(\fb)}.\end{equation}
We start from the injection
\[M^{(\fb)}_{\Lambda+\Lambda'}\hookrightarrow M^{(\fb)}_{\Lambda}\otimes K_{\Lambda'}^{(\fb)}.\]
Since the Zuckerman functor is left exact and commutes with taking tensor products with finite dimensional representions, see the subsequent Lemma \ref{Zuckcommtensor}, the application of the Zuckerman functor on the exact sequence above yields equation \eqref{injection}.

We use the identification of $\cR$ with the matrix elements of finite dimensional weight representations to study ${}^\lambda F(\fg)\cong \left(K_\lambda^{(\fb)}\right)^\ast$ for $\lambda\in\{\Lambda,\Lambda'\}$. We define a hermitian inner product on $K_\lambda^{(\fb)}$, denoted by $\langle\cdot,\cdot\rangle$, and we consider an orthonormal basis $\{e_j\}$. The matrix elements in ${}^\lambda F(\fg)$ are the ones of the form
\[U\mapsto \langle e_k| Uv^+_\lambda\rangle \quad\mbox{ for }\quad U\in\U(\fg), \]
$v^+_\lambda=e_1$ the highest weight vector of $K_\lambda^{(\fb)}$. The multiplication on $\U(\fg)^\ast$, as the dual Hopf algebra of $\U(\fg)$, is given in terms of the comultiplication on $\U(\fg)$. Therefore, the function $f$, which is the result of the multiplication of the functions $\langle e_k| \cdot v^+_\Lambda\rangle\in {}^\Lambda F(\fg)$ and $\langle f_l| \cdot v^+_{\Lambda'}\rangle\in {}^{\Lambda'}F(\fg)$, is given by
\[f(U)=\sum_j (-1)^{|U_j^{(2)}||f_l|} \langle e_k| U^{(1)}_j v^+_\Lambda\rangle\langle f_l| U^{(2)}_j v^+_{\Lambda'}\rangle, \]
using Sweedler's notation. This corresponds to the function given by the matrix elements of $K_\Lambda^{(\fb)}\otimes K_{\Lambda'}^{(\fb)}$ that are of the form $\langle e_k\otimes f_l|\cdot v_{\Lambda}^+\otimes v_{\Lambda'}^+\rangle$. The linear combinations of $e_k\otimes f_l$ that are generated by $\fg$-action on $v^+_{\Lambda}\otimes v^+_{\Lambda'}$ form $K_{\Lambda+\Lambda'}^{(\fb)}$, by equation~\eqref{injection}. This procedure shows that ${}^\Lambda F(\fg)\cdot {}^{\Lambda'} F(\fg)\supset {}^{\Lambda+\Lambda'}F(\fg)$.
\end{proof}

We introduce a symbol for the $\fg$-modules induced from simple finite dimensional $\fg_{\oa}$-modules. For each $\lambda\in \cP^+_{\oa}$, the finite dimensional $\fg$-module $C_\lambda$ is defined as 
\begin{equation}\label{indmod}C_\lambda\,=\ind^{\fg}_{\fg_{\oa}}L^{\oa}_\lambda=\U(\fg)\otimes_{\U(\fg_{\oa})}L^{\oa}_\lambda\cong\,\Hom_{\U(\fg_{\oa})}(\U(\fg),L^{\oa}_{\lambda}).
\end{equation}

\begin{theorem}
\label{UC}
The $\fg$-bimodule $\cR$ is given as a $\fg\times\fg_{\oa}$-module by
\begin{eqnarray*}
\cR\,\cong\, \Hom_{\U(\fg_{\oa})}(\U(\fg),\cR_{\oa})\,\cong\,\bigoplus_{\lambda\in \cP^+_{\oa}} C_{\lambda}\times \left(L_{\lambda}^{\oa}\right)^\ast.
\end{eqnarray*}
\end{theorem}
\begin{proof}
We have the following $\fg\times \fg_{\oa}$-module isomorphisms:
\begin{eqnarray*}\U(\fg)^\ast=\Hom_\C(\U(\fg),\C)&\cong&\Hom_\C(\U(\fg_{\oa})\otimes_{\U(\fg_{\oa})}\U(\fg),\C)\\
&\cong&\Hom_{\U(\fg_{\oa})}(\U(\fg),\Hom_\C(\U(\fg_{\oa}),\C)).\end{eqnarray*}
Since taking the left or right finite dual gives the same result according to Lemma~\ref{findual}, we take the right finite dual, which yields
\[\U(\fg)^\circ=\Hom_{\U(\fg_{\oa})}(\U(\fg),\U(\fg_{\oa})^\circ).\]

As an $\fh$-bimodule we thus have $\U(\fg)^0\cong \Lambda\fg_{\ob}\otimes \U(\fg_{\oa})^\circ$, so Lemma \ref{findual2} yields $\cR\cong \U(\fg)\otimes_{\U(\fg_{\oa})}\cR_{\oa}$. The second isomorphism then follows from the Peter-Weyl type decomposition \[\cR_{\oa}\cong \bigoplus_{\lambda\in\cP^+_{\oa}}L_{\lambda}^{\oa}\times \left(L_{\lambda}^{\oa}\right)^\ast\]
as $\fg_{\oa}\times\fg_{\oa}$-modules.
\end{proof}

We note that the isomorphism $\cR\cong \Hom_{\U(\fg_{\oa})}(\U(\fg),\cR_{\oa})$ is naturally linked to the construction of the sheaf of functions on a Lie supergroup $G$, starting from a Lie supergroup pair $(\fg,G_{\oa})$, $\cC^\infty(G)=\Hom_{\U(\fg_{\oa})}(\U(\fg),\cC^\infty(G_{\oa}))$.

The theorem above can be rewritten in terms of indecomposable projective modules in $\cF$.
\begin{corollary}
\label{projcover}
The $\fg$-bimodule $\cR$, as a $\fg\times\fg_{\oa}$-module, is isomorphic to
\begin{eqnarray*}
\cR\cong\,\bigoplus_{\Lambda\in \cP^+} P^\cF_{\Lambda}\times \left(L_{\Lambda}\right)^\ast.
\end{eqnarray*}
\end{corollary}
\begin{proof}
The projective module $C_\lambda$ can be decomposed into the indecomposable projective covers $P^\cF_\Lambda $ as $C_\lambda=\bigoplus_{\Lambda\in\cP^+}m_{\lambda\Lambda}P^\cF_\Lambda $ for certain constants $m_{\lambda\Lambda}\in\N$. The multiplicity is given by
\[m_{\lambda\Lambda}=\dim\Hom_{\fg}(C_\lambda,L_\Lambda)\]
since $\dim\Hom_{\fg}( P^\cF_{\Lambda'},L_\Lambda)=\delta_{\Lambda'\Lambda}$. Frobenius reciprocity then implies that 
\[m_{\lambda\Lambda}=\dim\Hom_{\fg_{\oa}}(L_\lambda^{\oa},\res^{\fg}_{\fg_{\oa}} L_\Lambda)=[\res^{\fg}_{\fg_{\oa}} L_\Lambda:L_\lambda^{\oa}].\]
Combining this with Theorem \ref{UC} implies the corollary.
\end{proof}
The following corollary generalises a reformulation of the classical Peter-Weyl decomposition.
\begin{corollary}
\label{HomRL}
For any integral dominant weight $\Lambda$, we have
\[\Hom_{\U(\fg)}(\cR,L_\Lambda)\cong L_\Lambda\qquad\mbox{and}\qquad \Ext^k_{\cF}(\cR,L_\Lambda)=0,\mbox{ for }k>0.\]
Moreover, the endofunctors of $\cF$, given by
$$ \cR\otimes_{\U(\fg)}-\qquad\mbox{and}\qquad \Hom_{\U(\fg)}(\cR,-), $$
are isomorphic to the identity.
\end{corollary}

For $\fg$ type I, we can use the $\Z$-gradation of $\fg=\fg_{-1}\oplus\fg_0\oplus\fg_1$ to obtain the description of $\U(\fg)$ as a $(\fg_0\oplus\fg_1)\times (\fg_{-1}\oplus\fg_0)$-module. 

\begin{theorem}
\label{UK}
For $\fg$ of type I, we have the isomorphism
\begin{eqnarray*}
\cR\cong \bigoplus_{\Lambda\in\cP^+}\left(K_\Lambda\right)^\vee\times \left(K_\Lambda\right)^\ast \quad \mbox{as}\quad  (\fg_0\oplus\fg_1)\times(\fg_0\oplus\fg_{-1})-\mbox{modules}.
\end{eqnarray*}
\end{theorem}
\begin{proof}
This is proved using the same ideas as in the proof of Theorem \ref{UC}, using the $\Z$-gradation $\fg=\fg_{-1}\oplus\fg_0\oplus\fg_1$.
\end{proof}

\subsection{The Zuckerman and Bernstein functor}

In his subsection we derive some general properties about the Zuckerman functor and its derived functors.
\begin{lemma}
\label{lemZuckR}
\label{BernO}
The Zuckerman functor can be represented as
\[S(M)\cong\Hom_{\U(\fg)}\left(\cR,M\right)\cong \cR\otimes_{\U(\fg)}M,\]
for $M\in\cC(\fg,\fl)$, where for the first equality invariants with respect to the left $\fg$-action on $\cR$ are taken and the right $\fg$-action on $\cR$ leads through duality to a left $\fg$-action on $\Hom_{\U(\fg)}\left(\cR,M\right)$. The derived functors therefore satisfy 
\begin{eqnarray*}
\cR_k S(M)=\Ext^k_{\cC(\fg,\fl)}(\cR,M)=H^k\left({\fg,\fl};\Hom_\C(\cR,M)\right)
\end{eqnarray*}
with $H^k({\fg,\fl};-)$ the relative algebra cohomology, see \cite{Hochschild}. Furthermore, if $M\in\cO$, we have
$$\cR_kS(M)\cong\Ext^k_{\cO}(\cR,M).$$

The Bernstein functor satisfies 
$$\cL_k \Gamma = \Ext^k_{\cC(\fg,\fl)}(-,\cR)^\ast,$$
which reduces to $\Ext^k_{\cO}(-,\cR)^\ast$ when restricted to $\cO$.
\end{lemma}
\begin{proof}
The identity $M\cong\U(\fg)\otimes_{\U(\fg)}M$ can be rewritten as $M\cong \Hom_{\U(\fg)}(\U(\fg)^\ast,M)$. The first result then follows from applying the Zuckerman functor and using Lemma~\ref{findual2}. The second representation of the Zuckerman functor follows similarly from the identity $M=\Hom_{\U(\fg)}(\U(\fg),M)$.

The reformulation $\cR_k S(M)=\Ext^k_{\cC(\fg,\fl)}(\cR,M)$ is an immediate consequence of the definition of $\cR_k S$ as the right derived functors of a functor $\cC(\fg,\fl)\to\cC(\fg,\fg_{\oa})$. The reformulation in terms of relative cohomology follows from the fact that the $(\fg,\fl)$-projective resolution of $\C$ in Section 5 of  \cite{Hochschild} is a projective resolution in the category $\cC(\fg,\fl)$.

The last reformulation follows from the fact that category $\cO$ is extension full in $\cC(\fg,\fh)$, see \cite{CouMaz2, MR2428237}.
\end{proof}

\begin{lemma}
\label{interZuck}
The right derived functors of the Zuckerman functors $S:\cC(\fg,\fl)\to\cC(\fg,\fg_{\oa})$ and $S_{\oa}:\cC(\fg_{\oa},\fl)\to\cC(\fg_{\oa},\fg_{\oa})$ in Definition \ref{defZuck} satisfy the following isomorphisms of functors:
\[\res^\fg_{\fg_{\oa}}\circ \cR_k S\cong \cR_k S_{\oa}\circ \res^\fg_{\fg_{\oa}}\quad\mbox{ and }\quad \cR_k S\circ \ind_{\fg_{\oa}}^{\fg}\cong\ind_{\fg_{\oa}}^{\fg}\circ \cR_k S_{\oa}.\]
\end{lemma}
\begin{proof}
The results follow from the combination of Lemma \ref{lemZuckR} and Theorem \ref{UC}. \end{proof}

From the combination of Lemma \ref{lemZuckR} applied to $\fg_{\oa}$ and Lemma \ref{interZuck}, we obtain the following corollary.
\begin{corollary}
\label{derzero}
The right derived functors of the Zuckerman functor $S:\cC(\fg,\fl)\to\cC(\fg,\fg_{\oa})$ satisfy $\cR_k S\cong 0$ for $k>\dim\fg_{\oa}\,-\,\dim\fl_{\oa}$.
\end{corollary}

\begin{lemma}
\label{Zuckcommtensor}
The Zuckerman functor and its derived functors commute with the functor corresponding to tensor multiplication with a finite dimensional $\fg$-module:
\[\cR_k S(-\otimes V)\cong \cR_k S(-)\otimes V,\]
for $k\in \N$ and $V$ a finite dimensional $\fg$-module. 
\end{lemma}
\begin{proof}
First, we prove this property for reductive Lie algebras and for $k=0$. It follows from
\begin{eqnarray*}
S_{\oa}(-\otimes L_{\mu}^{\oa})&\cong & \bigoplus_{\lambda\in\cP_{\oa}^+}L_{\lambda}^{\oa}\dim\Hom_{\fg_{\oa}}(L_{\lambda}^{\oa},-\otimes L_{\mu}^{\oa})\\
&=& \bigoplus_{\lambda\in\cP_{\oa}^+}L_{\lambda}^{\oa}\dim\Hom_{\fg_{\oa}}(L_{\lambda}^{\oa}\otimes \left(L_{\mu}^{\oa}\right)^\ast,-)
\end{eqnarray*}
and the fact that $L_{\lambda}^{\oa}\otimes \left( L_{\mu}^{\oa}\right)^\ast\cong\oplus_{\nu} c_{\lambda\nu}L_{\nu}^{\oa}$ implies $\oplus_{\lambda}c_{\lambda\nu} L_{\lambda}^{\oa}=L_{\nu}^{\oa}\otimes L_{\mu}^{\oa} $.

Now we turn to the case of Lie superalgebras. For $N$ a locally finite module, $N\otimes V$ is also locally finite. The natural morphism $S(M)\otimes V\hookrightarrow M\otimes V$ therefore leads to a morphism $$S(M)\otimes V\hookrightarrow S(M\otimes V).$$ On the other hand, Lemma \ref{interZuck} implies that $\Res\left(S(M)\otimes V\right)\cong\Res\left( S(M\otimes V)\right)$, so the injective isomorphism leads to a bijection $S(M\otimes V)\cong S(M)\otimes V$. The result for the derived functors follows from the property for $k=0$, the fact that tensoring with a finite dimensional module is an exact functor that maps injective modules to injective modules and the Grothendieck spectral sequence, see Section 5.8 in \cite{MR1269324}.
\end{proof}

\begin{corollary}
\label{tensorfin}
For a finite dimensional $\fg$-module $V$, we have
$$\Gamma_i(G/P,L_{\mu}(\fp)\otimes V)\cong \Gamma_i(G/P,L_{\mu}(\fp))\otimes V.$$
\end{corollary}
\begin{proof}
Proposition \ref{propBBW}(ii) implies 
$$\Gamma_i(G/P,L_{\mu}(\fp)\otimes V)\cong\cL_i\Gamma(\ind^{\fg}_{\fp}(L_{\mu}(\fp)\otimes V)).$$
Using the tensor identity and Lemma \ref{Zuckcommtensor} we thus obtain
$$\Gamma_i(G/P,L_{\mu}(\fp)\otimes V)\cong\cL_i\Gamma(\ind^{\fg}_{\fp}(L_{\mu}(\fp))\otimes V)\cong \cL_i\Gamma(\ind^{\fg}_{\fp}(L_{\mu}(\fp)))\otimes V,$$
which yields the result.
\end{proof}

\subsection{Analogues of BBW theory using twisting functors}
\label{subsecanal}

The functors $\cL_k\Gamma$ acting on Verma modules, which compute the cohomology groups of BBW theory, behave differently from the classical case if the highest weight is atypical. Contrary to the classical case, the Verma module with such an {\em integral dominant} highest weight is not projective in category $\cO$. We show that if we replace that Verma module by its projective cover, we do get classical results when the functors $\cL_k\Gamma$ act on it. According to Lemma \ref{TVerma} (or see \cite{MR2032059}), the non-dominant Verma modules for $\fg_{\oa}$ are obtained from the twisting functors acting on the dominant one. The following proposition is therefore an alternative extension of classical BBW theory to Lie superalgebras.

\begin{proposition}
Consider $\Lambda\in\cP^+$ an integral dominant weight and $w\in W$. We have the property
\[\cL_k\Gamma\left(T_wP^\cO_\Lambda \right)\,=\,\delta_{k,l(w)}\,P^\cF_\Lambda .\]
\end{proposition}
\begin{proof}
If $w=1$, then $\cL_k\Gamma(P^\cO_\Lambda)=0$ if $k>1$ by Lemma \ref{BernO}. The fact that $\Gamma(P^\cO_\Lambda)$ is projective in $\cF$ follows from the fact that the projective modules in $\cO$ are induced from $\fg_{\oa}$-modules while all modules which are projective in $\cF$ are direct summands of induced modules and Lemma \ref{interZuck}. The result $\Gamma(P^\cO_\Lambda)=P^\cF_\Lambda$ then follows from $\Top(P^\cO_\Lambda)=\Top\Gamma(P^\cO_\Lambda)=L_\Lambda$.

For $l(w)>1$ this follows from the combination of Lemma \ref{BernO}, Lemma \ref{IDMaz} and Lemma \ref{extVermafin} in the Appendix.
\end{proof}

According to Lemma \ref{TVerma}, another possibility to extend BBW from Lie algebras to Lie superalgebras is by replacing non-dominant Verma modules by the result of twisting functors acting on dominant Verma modules. This analogue of BBW theory is easier to describe than actuarial BBW theory, which follows from the following proposition
\begin{proposition}
\label{ZuckTwistVerma}
For $\Lambda\in\cP^+$ and $w\in W$, we have
$$\cL_k\Gamma( T_w M_\Lambda)=\begin{cases}K^{(\fb)}_\Lambda& \mbox{if $l(w)=k$}\\0& \mbox{if }l(w)>k\\ \Gamma_{k-l(w)}(G/B,L_{\Lambda}(\fb))=\cL_{k-l(w)}\Gamma(M_\Lambda)& \mbox{if }l(w) <k\end{cases}$$
\end{proposition}
\begin{proof}
This is a special case of Proposition \ref{extVermafin}.
\end{proof}
This analogue of BBW theory behaves therefore identical to BBW for Lie algebras when we take into account that $\Gamma_{\bullet}(G/B,L_{\Lambda}(\fb))$ for $\Lambda\in\cP^+$ can be non-zero in several degrees.

Even though it is not an analogue of BBW theory, the following result fits into the two propositions above.
\begin{lemma}
Consider $\mu\in\fh^\ast$ not integral dominant and $w\in W$. We have
$$\cL_k\Gamma (T_wP^{\cO}_\mu)=0\qquad\mbox{for all}\quad k\in\N.$$
\end{lemma}
\begin{proof}
If $w=1$ this follows immediately from Lemma \ref{BernO}. Since $P_{\cO}(\mu)$ has a standard filtration, the full result follows from Proposition \ref{extVermafin}.
\end{proof}


\section{Restriction to the $\fg_{\oa}$-module structure}
\label{secPenkov}

When we restrict to the $\fg_{\oa}$-module structure, the cohomology groups of BBW theory can be expressed in terms of the algebra cohomology of $\fu$ in finite dimensional $\fg$-modules (either indecomposable projective covers or the induced modules $C_\lambda$ from equation \eqref{indmod}). This is summarised in the following theorem, where for notational convenience we tacitly identify $\dim\Hom$ with $\Hom$. The second result is a rederivation of Corollary 1 in \cite{Gruson2}.

\begin{theorem}
\label{BBWKostant}
\label{BBWproj}
Consider a parabolic subalgebra $\fp$ of a basic classical Lie superalgebra $\fg$ such that the Levi subalgebra $\fl$ is of typical type. Consider an $\fl$-dominant $\mu\in\cP$, the $\fg$-modules $\Gamma_k(G/P,L_{\mu}(\fp))$ satisfy the relations
\begin{eqnarray*}
\res^\fg_{\fg_{\oa}}\Gamma_k(G/P,L_{\mu}(\fp))&=&\bigoplus_{\lambda\in\cP^+_{\oa}}\Hom_{\U(\fl)}\left(L_{\mu}(\fl),H^k(\fu,C_\lambda)\right) L_\lambda^{\oa}\quad\mbox{and}\\
\,[\Gamma_k(G/P,L_{\mu}(\fp)):L_\Lambda]&=& \Hom_{\U(\fl)}(L_\mu(\fl), H^k(\fu,P^\cF_\Lambda )),
\end{eqnarray*}
for any $\Lambda\in\cP^+$.
\end{theorem}
\begin{proof}
The first statement follows from the combination of Proposition \ref{propBBW}(iii) and Theorem \ref{UC}.

Corollary \ref{projcover} and Proposition \ref{propBBW}(iii) imply that 
\[\ch\Gamma_k(G/P,L_{\mu}(\fp))=\sum_{\Lambda\in\cP^+} \Hom_{\U(\fl)}(L_\mu(\fl), H^k(\fu,P^\cF_\Lambda ))\ch L_\Lambda.\]
Since the character of a finite dimensional weight module completely determines the multiplicities in the Jordan-H\"older decomposition, the second property follows.
\end{proof}

The following theorem shows that the Kostant cohomology of projective modules in $\cF$ appears only in finitely many degrees, this is not true for arbitrary modules, see e.g. \cite{BGG}. Furthermore, it presents $\fl$-modules which serve as an upper bound for the cohomology groups. This could also be obtained through the equality
\begin{equation}\label{ExtRes}\Hom_{\fh}(\C_\mu,H^k(\fn,C_\lambda))=\Ext^k_{\cO}(M(\mu),C_\lambda)=\Ext^k_{\cO_{\oa}}(\res^{\fg}_{\fg_{\oa}}M(\mu),L_\lambda^{\oa})\end{equation}
and the standard filtration of $\res^{\fg}_{\fg_{\oa}}M(\mu)$ by Verma modules of $\fg_{\oa}$, but underneath we take a more constructive approach which leads to an explicit construction of the maps of this complex to compute the cohomology.

\begin{theorem}
\label{infocohomC}
The cohomology groups $H^k(\fu,C_\lambda)$ are isomorphic to the homology groups $H_k(\overline{\fu},C_\lambda)$ and can be computed as the homology of a complex of $\fl$-modules of the form
\[0\to\Lambda\fg_1\otimes H^d({\fu}_{\oa},L_\lambda^{\oa})\to\cdots\to\Lambda\fg_1\otimes H^j({\fu}_{\oa},L_\lambda^{\oa})\to\cdots\to\Lambda\fg_1\otimes H^0({\fu}_{\oa},L_\lambda^{\oa})\to0,\]
with $d=\dim\fu_{\oa}$.
\end{theorem}
\begin{proof}
The equivalence of the $\fu$-cohomology and $\overline{\fu}$-homology follows from the general relation $H^k(\fu,V)=H_k(\overline{\fu},V^\vee)^\vee$ see e.g. Remark 4.1 in \cite{BGG}. Since all finite dimensional $\fl$-weight representations and the induced $\fg$-module $C_\lambda$ are self-dual with respect to $\vee$, the twisted duals can be ignored.

The homology groups $H_j(\overline{\fu}_{\oa},L_\lambda^{\oa})\cong H^j(\fu_{\oa}, L_\lambda^{\oa})$ of \cite{MR0142696} can be obtained from a projective resolution of $L_\lambda^{\oa}$ in the category of $\overline{\fu}_{\oa}$-modules, which can even be written as a resolution of $\fg_{\overline{0}}$-modules. These resolutions correspond to Lepowsky's generalisation of the Bernstein-Gelfand-Gelfand resolutions for the reductive Lie algebra $\fg_{\oa}$ with parabolic subalgebra $\fp_{\oa}$, see \cite{lepowsky}. This is an exact complex of $\fg_{\oa}$-modules of the form
\begin{eqnarray*}
&&0\rightarrow \U(\fg_{\oa})\otimes_{\U(\fp_{\oa})}H_d(\overline{\fu}_{\oa},L^{\oa}_\lambda) \to\cdots\to \U(\fg_{\oa})\otimes_{\U(\fp_{\oa})}H_j(\overline{\fu}_{\oa},L^{\oa}_\lambda)\\
&&\to\cdots\to  \U(\fg_{\oa})\otimes_{\U(\fp_{\oa})}H_0(\overline{\fu}_{\oa},L^{\oa}_\lambda)\to L^{\oa}_\lambda\to0.\end{eqnarray*}
Applying the exact functor $\U(\fg)\otimes_{\U(\fg_{\overline{0}})}-$: $\cC(\fg_{\oa},\fl)\to \cC(\fg,\fl)$ we obtain an exact complex of $\fg$-modules, which is a resolution by free $\overline{\fu}$-modules, so it can be used to compute the right derived functors of the left exact contravariant functor $\Hom_{\overline{\fu}}(-,\C)$ acting on $C_\lambda$. Since \[\Hom_{\overline{\fu}}(\U(\fg)\otimes_{\U(\fp_{\oa})}M,\C)=\Hom_{\overline{\fu}}(\U(\overline{\fu})\otimes\Lambda \fg_1\otimes M,\C)=\left(\Lambda \fg_1\otimes M\right)^\ast,\]
for an arbitrary $\fp_{\oa}$-module $M$, the homology groups $\Ext^k_{\overline{\fu}}(C_\lambda,\C)$ can be calculated from the complex
\begin{eqnarray*}
&&0\rightarrow \left(\Lambda \fg_1\otimes H_0(\overline{\fu}_{\oa},L^{\oa}_\lambda)\right)^\ast\to\cdots\to \left(\Lambda \fg_1\otimes H_j(\overline{\fu}_{\oa},L^{\oa}_\lambda)\right)^\ast\\
&&\to\cdots\to \left(\Lambda \fg_1\otimes H_d(\overline{\fu}_{\oa},L^{\oa}_\lambda)\right)^\ast\to0.\end{eqnarray*}
The theorem then follows from the observation $\Ext^k_{\overline{\fu}}(-,\C)\cong H_k(\overline{\fu},-)^\ast$, see e.g. Lemmata 4.6 and 4.7 in \cite{BGG}.
\end{proof}

This leads to the same results as were obtained by Gruson and Serganova through geometric methods in \cite{MR2734963}.
\begin{corollary}
\label{Penkovresult}
The cohomology groups $\Gamma_k(G/P,L_{\lambda}(\fp))$ satisfy
\begin{eqnarray*}
\res^\fg_{\fg_{\oa}}\Gamma_k(G/P,L_{\lambda}(\fp))&\le&\Gamma_k(G_{\oa}/P_{\oa},\Lambda\fg_{-1}\otimes L_{\lambda}(\fp_{\oa}))\qquad\mbox{and}\\
\sum_{k=0}^\infty(-1)^k \ch \Gamma_k(G/P,L_{\lambda}(\fp))&=&\sum_{k=0}^\infty(-1)^k \ch \Gamma_k(G_{\oa}/P_{\oa},\Lambda\fg_{-1}\otimes L_{\lambda}(\fp_{\oa})),
\end{eqnarray*}
where the $\fp_{\oa}$-module structure on $\Lambda\fg_{-1}$ is given by adjoint $\fl$-action and trivial $\fu_{\overline{0}}$-action.
\end{corollary}
\begin{proof}
The first property follows from the combination of Theorem \ref{BBWKostant} and Theorem~\ref{infocohomC}, which implies
\begin{eqnarray*}
\res^\fg_{\fg_{\oa}}\Gamma_k(G/P,L_{\mu}(\fp))&=&\bigoplus_{\lambda\in\cP^+_{\oa}}\Hom_{\U(\fl)}\left(L_{\mu}(\fl),H^k(\fu,C_\lambda)\right) L_\lambda^{\oa}\\
&\le&\bigoplus_{\lambda\in\cP^+_{\oa}}\Hom_{\U(\fl)}\left(\Lambda\fg_{-1}\otimes L_{\mu}(\fl),H^k(\fu_{\oa},L^{\oa}_{\lambda})\right) L_\lambda^{\oa}.
\end{eqnarray*}
The classical BBW theorem of \cite{MR0229257, MR0563362} then yields the result.

The second property follows from similarly from Theorems \ref{BBWKostant} and \ref{infocohomC}, by applying the Euler-Poincar\'e principle. 
\end{proof}

Theorem \ref{infocohomC} implies that the cohomology groups in equation \eqref{ExtRes}, $$\Ext^\bullet_{\cO_{\oa}}(\Res M(\mu),\cR_{\oa}),$$ can be computed as the cohomology of a complex with spaces of chains $$\bigoplus_{\gamma\in\Gamma^+}\Ext^\bullet_{\cO_{\oa}}(M_{\oa}(\mu-\gamma),\cR_{\oa}).$$
Note that $\Res M(\mu)$ has a standard filtration, with the Verma modules for $\fg_{\oa}$ appearing in the equation above.

\begin{corollary}
\label{degvanish}
For any finite dimensional $V\in\cC(\fp,\fl)$, we have
\[\Gamma_k(G/P,V)=0\qquad\mbox{for} \quad k> \dim\fu_{\oa}.\]
 \end{corollary}
 \begin{proof}
If $V$ is of the form $L_{\mu}(\fp)$ for $\mu\in\fh^\ast$ an $\fl$-integral dominant weight, this is an immediate consequence of Corollary \ref{Penkovresult}.
 
An arbitrary such module $V$ has a finite filtration by irreducible $\fp$-modules. The statement can then be proved by induction on the filtration length.
 \end{proof}

By applying Corollary \ref{Penkovresult}, one can reobtain Lemma 3, Lemma 5, Corollary 2 and Proposition 1 in \cite{MR2734963}. Since we will need the results in the sequel, we state (a slightly stricter version of) Lemma~3 and Proposition 1 of \cite{MR2734963}.
\begin{lemma}
\label{restrHW}
If for $\Lambda\in\cP^+$, $L_\Lambda$ occurs in $\Gamma_k(G/P,L_\mu(\fp))$ as a subquotient, then
\[\Lambda\in w(\mu+\rho)-\rho-\Gamma^+,\quad \mbox{for some } w\in W\mbox{  of length }k,\]
such that $w^{-1}\in W^1(\fl_{\oa})$.
\end{lemma}
\begin{proof}
Corollary \ref{Penkovresult} and the classical BBW theorem in \cite{MR0229257, MR0563362} imply that 
\[u\circ\Lambda\in\mu-\Gamma^+\] 
for some $u\in W^1(\fl_{\oa})$ of length $k$, then we apply equation \eqref{MusGS}.
\end{proof}

\begin{lemma}
\label{Euler}
The Euler characteristic of the cohomology groups of BBW theory satisfies
\[\sum_{k=0}^\infty(-1)^k \ch \Gamma_k(G/P,L_{\lambda}(\fp))=\frac{\prod_{\gamma\in\Delta_{\ob}^+}(1+e^{-\gamma})}{\prod_{\alpha\in\Delta_{\oa}^+}(1-e^{-\alpha})}\sum_{w\in W}(-1)^{l(w)}e^{w(\lambda+\rho)-\rho}.\]
\end{lemma}
\begin{proof}
We prove this property for $\fp=\fb$. The property for general $\fp$ follows similarly by applying standard combinatorial equalities, but also from the case $\fb=\fp$ and Theorem \ref{thmBBW}(i).

The classical BBW theorem implies
\begin{eqnarray*}
\sum_{k=0}^\infty(-1)^k \ch \Gamma_k(G_{\oa}/B_{\oa},L_{\mu}(\fb_{\oa}))&=&\frac{\sum_{w\in W}(-1)^{l(w)}e^{w(\mu+\rho_{\oa})}}{\prod_{\alpha\in \Delta_{\oa}^+}(e^{\alpha/2}-e^{-\alpha/2})}.
\end{eqnarray*}
The second statement in Corollary \ref{Penkovresult} therefore implies
\begin{eqnarray*}
\sum_{k=0}^\infty(-1)^k \ch \Gamma_k(G/B,L_{\lambda}(\fb))&=&\frac{\sum_{w\in W}(-1)^{l(w)}w\left(e^{\rho_{\oa}+\lambda}\prod_{\gamma\in\Delta^+_{\ob}}(1+e^{\gamma})\right)}{\prod_{\alpha\in \Delta_{\oa}^+}(e^{\alpha/2}-e^{-\alpha/2})},
\end{eqnarray*}
which yields the proposed formula.
\end{proof}

As in \cite{Gruson2} we denote the Euler characteristic of Lemma \ref{Euler} by 
\begin{equation}\label{formEuler}\cE(\lambda)=\frac{\prod_{\gamma\in\Delta_{\ob}^+}(1+e^{-\gamma})}{\prod_{\alpha\in\Delta_{\oa}^+}(1-e^{-\alpha})}\sum_{w\in W}(-1)^{l(w)}e^{w(\lambda+\rho)-\rho}.\end{equation}


\section{Simple reflections}
\label{specseq}

Theorem \ref{thmBBW}(i) implies that, to obtain BBW theory for arbitrary parabolic subalgebras, with a Levi subalgebra of typical type, we only need to solve the case where the parabolic subalgebra is the Borel subalgebra.

In Proposition 6 in \cite{MR0229257}, Demazure showed how BBW theory for Lie algebras can be reduced to the case of $\mathfrak{sl}(2)$ by changing from one Borel subalgebra to another one through a simple reflection. This was also obtained by Enright and Wallach in Lemma 6.2 in \cite{MR0563362} by a different approach. In Subsection \ref{evenrefl} we show that the same idea can be used for Lie superalgebras. This was obtained earlier by Penkov in \cite{MR0957752} and by Santos in \cite{MR1680015}, through reducing to $\mathfrak{sl}(2)$ or $\mathfrak{osp}(1|2)$. Here we use a different technique, based on the properties of twisting functors developed in Appendix \ref{aptwist}, which is motivated by the insight it provides in a broader range of possibilities. 

In Subsection \ref{oddrefl} we explore what happens when two Borel subalgebras are connected through a reflection corresponding to a simple isotropic (odd) root, which corresponds to a reduction to $\mathfrak{sl}(1|1)$.

One consequence of these results is a complete solution of BBW theory for (i) basic classical Lie superalgebras of type I with distinguished Borel subalgebra and (ii) BBW theory for the typical blocks. These results are well-known, see e.g. \cite{MR1680015, MR2734963, MR0957752, MR2059616}, so we only mention this briefly in Subsection \ref{subtyptypeI}

\subsection{Even reflection}
\label{evenrefl}
\begin{theorem}
\label{thmsimple}
For $\alpha\in\Delta^{+}$ a simple non-isotropic root and $\mu\in\cP$, we have
\[\Gamma_k(G/B,L_\mu(\fb))=\Gamma_{k-1}\left(G/B,L_{s_{\alpha}\cdot\mu}(\fb)\right)\quad\mbox{ if }\quad\langle \alpha^\vee,\mu\rangle <0.\]
\end{theorem}
\begin{proof}
Because of Proposition \ref{propBBW}(ii) this amounts to proving that
\begin{eqnarray*}
\Ext_{\cO}^k(M_\mu,\cR)&\cong& \Ext^{k-1}_{\cO}(M_{s_\alpha\cdot\mu},\cR)
\end{eqnarray*}
holds for any simple non-isotropic root $\alpha$ with $\langle \alpha^\vee,\mu\rangle <0$. This is a consequence of Lemma \ref{TVerma} for $\lambda=s_\alpha\cdot\mu$ and Lemma \ref{IDMaz} in Appendix \ref{aptwist}.
\end{proof}

\begin{remark}
\label{Kostevenrefl}
The proof of the result above can immediately be extended to the property that if $\langle \mu,\alpha^\vee\rangle<0$ holds, we have
$$\Hom_{\fh}(\C_\mu,H^k(\fn,V))\cong \Hom_{\fh}(\C_{s_\alpha\cdot\mu},H^{k-1}(\fn,V))$$
for any locally finite $\fg$-module $V$. Alternatively, this can be derived from the corresponding property for $\mathfrak{sl}(2)$ or $\mathfrak{osp}(1|2)$ depending on whether $\alpha$ or $\alpha/2$ is simple in $\Delta^+$, using a Hochschild-Serre spectral sequence, as is done in proposition 3.9 in~\cite{MR1680015}.
\end{remark}

\subsection{Odd reflection}

\label{oddrefl}

Consider two Borel subalgebras $\fb$ and $\widetilde{\fb}$ of $\fg$ with $\fb_{\oa}={\widetilde{\fb}}_{\oa}$, then they can be linked to each other by odd reflections, see Theorem 3.1.3 in \cite{MR2906817}. We say that the ordered set of odd roots $\beta_1,\cdots,\beta_p$ takes $\fb$ to $\widetilde{\fb}$ if there are $p+1$ systems of positive roots $\{S_j,j=0,\cdots ,p\}$ such that $S_0=\Delta^+$ and $S_p=\widetilde{\Delta}^+$ are the ones corresponding to $\fb$ and $\widetilde{\fb}$ and $S_j=S_{j-1}\backslash\beta_j\,\cup\,-\beta_j$.
\begin{lemma}
\label{changeBorel}
Consider two Borel subalgebras $\fb$ and $\widetilde{\fb}$ of $\fg$ with $\fb_{\oa}={\widetilde{\fb}}_{\oa}$ and $\beta_1,\cdots,\beta_p$ the ordered set of odd roots which take $\fb$ to $\widetilde{\fb}$. If $\langle \beta_j,\mu-\beta_1-\cdots-\beta_{j-1}\rangle\not=0$ for $j=1,\cdots,p$, then it holds that
\[\Gamma_k(G/B, L_{\mu}(\fb))\cong\Gamma_k(G/\widetilde{B},L_{\mu+\rho-\widetilde{\rho}}(\widetilde{\fb}))\]
for every $k\in\N$. 
\end{lemma}
\begin{proof}
It suffices to prove that if $\langle \gamma,\mu\rangle=0$ for an isotropic simple root $\gamma$ in $\Delta^+$,
\[\Gamma_k(G/B, L_{\mu}(\fb))\cong\Gamma_k(G/\widetilde{B},L_{\mu-\gamma}(\widetilde{\fb}))\]
holds for $\widetilde{\fb}=(\fb\backslash\fg_{\gamma})\oplus \fg_{-\gamma}$. The result thus follows from Theorem \ref{thmBBW}(ii) and the fact $M^{(\fb)}_{\mu}\cong M^{(\widetilde{\fb})}_{\mu-\gamma}$ if and only if $\langle \gamma,\mu\rangle=0$, see e.g. the proof of Lemma 2.3 in~\cite{CouMaz}.
\end{proof}

\begin{remark}
In particular, if $\mu$ is typical, the condition $\langle \beta_j,\mu-\beta_1-\cdots-\beta_{j-1}\rangle\not=0$ is always satisfied since for $\gamma$ a simple isotropic root, $\langle \gamma,\rho\rangle=0$ holds.
\end{remark}

\begin{corollary}
\label{thmnotsimple}
Let $\alpha$ be a non-isotropic simple root in $\Delta^+$. If $\langle \alpha^\vee,\mu\rangle < 0$ and
\[\mbox{for } j=1,\cdots,p:\quad\begin{cases}\langle\beta_j,\mu-\beta_1-\cdots-\beta_{j-1}\rangle \not=0\\\langle\beta_j,s_{\alpha}(\mu+\rho)-\rho-\beta_1-\cdots-\beta_{j-1}\rangle\not=0,\end{cases}\]
with $\beta_1,\cdots,\beta_p$ the ordered set of odd roots changing the Borel algebra $\fb$ into one where $\alpha$ is a simple root, we have
\[\Gamma_k(G/B,L_\mu(\fb))=\Gamma_{k-1}\left(G/B,L_{s_{\alpha}\cdot\mu}(\fb)\right).\]
\end{corollary}
\begin{proof}
There is always a Borel subalgebra $\widetilde{\fb}$ with $\widetilde{\fb}_{\oa}=\fb_{\oa}$ where $\alpha$ (or $\alpha/2$) is simple. The combination of Lemma \ref{changeBorel} and Theorem \ref{thmsimple} for $\widetilde{\fb}$ yields
\[\Gamma_k\left(G/B,L_\mu(\fb)\right)=\Gamma_{k-1}\left(G/\widetilde{B},L_{s_{\alpha}(\mu+\rho)-\widetilde{\rho}}(\widetilde{\fb})\right).\]
The result then follows from Lemma \ref{changeBorel} if $\langle-\beta_i, s_{\alpha}(\mu+\rho)-\widetilde{\rho}+\beta_p+\cdots+\beta_{i+1}\rangle$ is zero for $i\in\{1,\cdots,p\}$, which can be rewritten as the proposed condition on $s_{\alpha}\cdot\mu$.
\end{proof}

For completeness we state what happens in case the condition in Lemma \ref{changeBorel} is not satisfied for two adjacent Borel subalgebras.
\begin{lemma}
\label{2resol}
Consider an isotropic simple root $\gamma$ in $\Delta^+$ and $\widetilde{\fb}$ the Borel subalgebra created from $\fb$ by the odd reflection of $\gamma$. There are $\fg$-modules $\{A_j,j\ge 0\}$ and $\{B_j,j\ge 0\}$ in $\cC(\fg,\fg_{\oa})$, such that there are two exact sequences of the form
{\small
\begin{eqnarray*}
\quad\cdots \to A_{k+1}\to B_k\to \Gamma_k(G/B, L_{\mu}(\fb))\to A_k\to\cdots\to B_0 \to\Gamma_0(G/B, L_{\mu}(\fb))\to  A_0\to 0\\
\cdots \to A_k\to \Gamma_k(G/\widetilde{B},L_{\mu-\gamma}(\widetilde{\fb}))\to B_k\to\cdots\to A_0 \to\Gamma_0(G/B, L_{\mu-\gamma}(\widetilde{\fb}))\to  B_0\to 0.
\end{eqnarray*}}
\end{lemma}


\begin{proof}
We denote a nonzero root vector with weight $\gamma$ by $X_\gamma$ and corresponding negative root vector by $Y_\gamma$. If $\langle \mu,\gamma\rangle=0$, then $M^{(\fb)}_{\mu}$ is no longer a Verma module with respect to $\widetilde{\fb}$, see e.g. the proof of Lemma 2.3 in \cite{CouMaz}. However, there are $\fg$-modules $I$ and $K$ such that we have short exact sequences
$$K\hookrightarrow M^{(\fb)}_\mu\tto I\qquad\mbox{and}\qquad I\hookrightarrow M^{(\widetilde{\fb})}_{\mu-\gamma}\tto K,$$
with $K$ the subalgebra of $M^{(\fb)}(\mu)$ generated by the vector of weight $\mu-\gamma$.

The result then follows form applying the right exact functor $\Gamma$ to the short exact sequences, identifying $A_k=\cL_k\Gamma(I)$ and $B_k=\cL_k\Gamma(K)$ and Theorem \ref{thmBBW}(ii).
\end{proof}


We remark that the $A_k$ and $B_k$ can be interpreted as cohomology groups of the form $\Gamma_k(G/P_\alpha,L_{\nu}(\fp_\gamma))$, for $\fp_\gamma$ the parabolic subalgebra defined as $\fp_\gamma=\fg_{-\gamma}\oplus\fb$, with Levi subalgebra isomorphic to $\fh+\mathfrak{gl}(1|1)$.

\begin{remark}
Similar to the case of even reflections, the proof of the results in this subsection extend immediately to the statement that for a locally finite $\fg$-module $V$,
\begin{equation}\label{gl11}\Hom_{\fh}(\C_\mu,H^k(\fn,V))\cong \Hom_{\fh}(\C_{\mu-\gamma}, H^k(\tau_{\gamma}(\fn),V),\end{equation}
if for a simple isotropic root $\gamma$, $\langle \mu,\gamma\rangle\not=0$ holds, with $\tau_{\gamma}(\fn)=(\fn\,\backslash\, \fg_{\gamma})\oplus \fg_{-\gamma}$. An alternative derivation of this result is through reducing to the corresponding result for $\mathfrak{gl}(1|1)$ using a Hochschild-Serre spectral sequence. The condition $\langle \mu,\gamma\rangle\not=0$ then assures that typical finite dimensional $\mathfrak{gl}(1|1)$ representations are considered, which are $\fg_{\gamma}$-free.

If the extra assumption is made that the $\mathfrak{gl}(1|1)$-modules $H^{j}(\fn(\gamma),V)$, with $\fn(\gamma)=\fn\backslash \fg_{\gamma}$, are $\fg_\gamma$ and $\fg_{-\gamma}$-free (which corresponds to to being projective in the category of finite dimensional $\mathfrak{gl}(1|1)$-modules) we have the equality \eqref{gl11} without the condition $\langle \mu,\gamma\rangle\not=0$. However, the condition that $V\in\cF$ is projective is not sufficient for this.
\end{remark}

\subsection{Applications of even and odd reflections}
\label{subtyptypeI}
\begin{theorem}[Theorem 5.2 of \cite{MR2059616}]
Consider $\fg$ a basic classical Lie superalgebra of type I with distinguished Borel subalgebra $\fb^d$.
\begin{itemize}
\item If $\lambda$ is regular, there exists a unique $w\in W$ such that $\Lambda=w\cdot\lambda\in\cP^+$ and
\[\Gamma_k(G/B^d,L_{\lambda}(\fb^d))=\begin{cases}K_{\Lambda}&\mbox{if} \quad l(w)=k\\ 0 &\mbox{if}\quad  l(w)\not=k  \end{cases}.\]
\item If $\lambda$ is singular, $\Gamma_k(G/B^d,L_{\lambda}(\fb^d))=0$.
\end{itemize}
\end{theorem}
\begin{proof}
This follows from Corollary \ref{degvanish}, Theorem \ref{thmsimple} and Lemma \ref{zerocohomKac}. An alternative proof is to use Corollary 8.1 in \cite{BGG} and Theorem \ref{UK}.
\end{proof}

Comparing this result to Proposition \ref{propBBW}(iii) then yields the following corollary.

\begin{corollary}
For $\fg$ a basic classical Lie superalgebra of type I with distinguished system of positive roots, the $\fn$-cohomology of $\cR$ is given by
\begin{eqnarray*}
H^k(\fn,\cR)&=&\bigoplus_{\Lambda\in\cP^+}\bigoplus_{w\in W(k)} \C_{w\cdot\Lambda}\times \left(K_\Lambda\right)^\ast\quad\mbox{as $\fh\times\fg$-modules.}
\end{eqnarray*}
\end{corollary}

As a consequence of BBW theory for basic classical Lie superalgebras of type I, the Kostant cohomology of projective modules in $\cF$ is known. These could also be calculated immediately from the fact that they are finite dimensional modules which are $\fg_{1}$-free.

\begin{corollary}
\label{KosProjI}
For $\fg$ a basic classical Lie superalgebra of type I with standard Borel subalgebra $\fb^d$, the Kostant cohomology of projective covers in $\cF$ satisfies
\[\ch H^k(\fn,P^\cF_\Lambda )=\sum_{w\in W(k)}w\cdot\ch\left(H^0(\fn,P^\cF_\Lambda )\right).\]
Here $H^0(\fn,P^\cF_\Lambda )$ can be described by
\[\Hom_\fh(\C_\lambda,H^0(\fn,P^\cF_\Lambda ))=[K_\lambda:L_\Lambda]\]
if $\lambda\in\cP^+$ and $\Hom_\fh(\C_\lambda,H^0(\fn,P^\cF_\Lambda ))=0$ otherwise. The multiplicities $[K_\lambda:L_\Lambda]$ have been calculated in \cite{Brundan}. 
\end{corollary}


The following theorem corresponds to Theorem 1 in \cite{MR0957752}. It can be obtained from the combination of Theorem \ref{thmBBW}(i) and Corollary \ref{thmnotsimple}, but is also a consequence of the combination of Lemma \ref{restrHW} and Corollary \ref{corcentr}. Here we use the results on twisting functors to obtain a very short proof.

\begin{theorem}
\label{BBWtypical}
Consider $\fg,\fp,\fl,\fu,\fh$ as in the preliminaries and $\lambda\in\cP$ typical and $\fl$-dominant.
\begin{itemize}
\item If $\lambda$ is regular, there exists a unique $w\in W$ such that $\Lambda=w\cdot\lambda\in\cP^+$ and
\[\Gamma_k(G/B^d,L_{\lambda}(\fb^d))=\begin{cases}K_{\Lambda}&\mbox{if} \quad l(w)=k\\ 0 &\mbox{if}\quad  l(w)\not=k  \end{cases}.\]
\item If $\lambda$ is singular, $\Gamma_k(G/B^d,L_{\lambda}(\fb^d))=0$.
\end{itemize}\end{theorem}
\begin{proof}
This is an immediate consequence of Proposition \ref{ZuckTwistVerma}.
\end{proof}

\begin{corollary}
\label{Kostanttypical}
The $\fn$-cohomology in typical simple $\fg$-modules satisfies
\begin{eqnarray*}
H^k(\fn,L_\Lambda)&=&\bigoplus_{w\in W(k)}\C_{w(\Lambda+\rho)-\rho}.
\end{eqnarray*}
\end{corollary}
\begin{proof}
This can be obtained from Theorem \ref{BBWtypical} and Theorem \ref{BBWproj} since the block in the category of finite dimensional representations corresponding to a typical character is semisimple and therefore $P_\Lambda\cong L_\Lambda$.\end{proof}

For $\fg=\mathfrak{osp}(1|2n)$, all blocks are typical. For this case these results can also be obtained in the reversed order. A direct calculation can be applied to reduce the Kostant cohomology in Corollary \ref{Kostanttypical} to that of $\mathfrak{so}(2n+1)$, from which the BBW result follows. This is done in \cite{osp12n}.

\begin{remark}
\label{MorGorelik}
Yet another way to prove BBW theory for the (strongly) typical blocks is the Morita equivalence in \cite{MR1862800}. This equivalence of categories maps the BGG resolutions for $\fg_{\oa}$ to BGG resolutions for $\fg$. From these the Kostant cohomology can be calculated and the BBW theorem follows.
\end{remark}


\section{BBW theory for generic weights}
\label{secgen}

In this section we discuss BBW theory for generic weights, see Definition \ref{defgeneric}. For $\widetilde{\Gamma}$-generic weights, the star action of Section 8.1 in \cite{CouMaz} becomes uniquely defined and leads to an action of the Weyl group as proved in Theorem 8.10 in \cite{CouMaz}. This is a deformation of the usual $\rho$-shifted action of the Weyl group, of which the orbits only coincide with the undeformed orbits in case $\fg$ is of type I and $\fb$ is the distinguished Borel subalgebra. Our main result is the following theorem.

\begin{theorem}
\label{maingen}
Consider a basic classical Lie superalgebra $\fg$ with arbitrary Borel subalgebra $\fb$. If $\Lambda\in\cP^+$ is ${\Gamma}^+$-generic and $w\in W$, we have the following properties
\begin{itemize}
\item $\Gamma_k(G/B, L_{w\cdot \Lambda}(\fb))\,=\,\delta_{k,l(w)}\,\, K^{(\fb)}_\Lambda[w],\,\quad$ with
\item $\ch K_\Lambda^{(\fb)}[w]=\ch K_\Lambda^{(\fb)}\quad$ and
\item $K_\Lambda^{(\fb)}[w]\tto L^{(\fb)}_{w^{-1}\ast w\cdot\Lambda}$ if $\Lambda$ is generic.
\end{itemize}
\end{theorem}
The first two properties are known, see e.g. \cite{MR1309652, MR1036335}, the thrid one is new. An alternative formulation of Theorem \ref{maingen} is as follows.
\begin{remark}
 Consider $\mu\in\cP$ generic weight, then there is a unique $w\in W$ such that $\Lambda_1=w\cdot\mu$ is integral dominant. For this $w$, the weight $\Lambda_2=w\ast^{\fb}\mu$ is also integral dominant and we have
\begin{itemize}
\item $\Gamma_k(G/B, L_{\mu}(\fb))\,=\,\delta_{k,l(w)}\,\, K^{(\fb)}_{\Lambda_1}[w],\,\quad$ with
\item $\ch K_{\Lambda_1}^{(\fb)}[w]=\ch K_{\Lambda_1}^{(\fb)}\quad$ and
\item $K_{\Lambda_1}^{(\fb)}[w]\tto L^{(\fb)}_{\Lambda_2}$.
\end{itemize}
\end{remark}
This shows how the usual $\rho$-shifted action and the star action play an complementary role in the description of BBW theory for Lie superalgebras. 
\begin{proposition}
\label{completion}
If $\Lambda$ is a $\Gamma^+$-generic integral dominant weight and $w\in W$, the cohomology groups of BBW theory satisfy
$$\Gamma_k(G/B,L_{w\cdot\Lambda}(\fb))\cong\delta_{k,l(w)}\Gamma(G_{w^{-1}} M_{w\cdot\Lambda}).$$
\end{proposition}
\begin{proof}
According to Lemma \ref{restrHW}, we can only have $[\Gamma_k(G/B,L_{w\cdot\Lambda}):L_{\Lambda'}]\not=0$ for an integral dominant weight $\Lambda'$, if there is a an element $u\in W(k)$ such that $\Lambda'\in u\cdot w\cdot\Lambda-\Gamma^+$, or in other words
$$uw\circ(\Lambda-\Gamma^+)\cap\cP^+\not=0.$$
Based on the definition of $\Gamma^+$-genericness, this can only occur if $uw=1$, so in particular $k=l(w)$. This proves that the cohomology is contained in one degree.

The cohomology group for $k=l(w)$ is a consequence of Theorem \ref{thmBBW}(iii) and Corollary \ref{extcomple} in Appendix A.
\end{proof}

Now we can prove the main theorem.
\begin{proof}[Proof of Theorem \ref{maingen}]
According to Proposition \ref{completion}, we only need to study the module $\Gamma(G_{w^{-1}}M_{w\cdot\Lambda})$. Lemma \ref{interZuck} and Corollary \ref{resVerma} imply $\Res \Gamma( G_{w^{-1}}M_{w\cdot\Lambda})=\Res K^{(\fb)}_{\Lambda}$. The remainder then follows from Proposition \ref{Berncompl}.
\end{proof}

\begin{corollary}
\label{KoProGen}
(i) Consider $\Gamma^+$-generic $\Lambda_1\in\cP^+$ and arbitrary $\Lambda_2\in\cP^+$, then
$$\Hom_{\fh}(\C_{w\cdot\Lambda_1}, H^k(\fn, P_{\Lambda_2}^{\cF}))=\delta_{k,l(w)}\Hom_{\fh}(\C_{\Lambda_1},H^0(\fn,P_{\Lambda_2}^{\cF})).$$

(ii) If $\Lambda\in\cP^+$ is $\widetilde{\Gamma}$-generic, then $$\ch H^k(\fn,P_\Lambda^{\cF})=\bigoplus_{w\in W} w\cdot\ch H^0(\fn,P_\Lambda^{\cF}).$$
\end{corollary}
\begin{proof}
The first property is a reformulation of Theorem \ref{maingen} through Theorem \ref{BBWproj}.

If $\Lambda$ is $\widetilde{\Gamma}$-generic and if $\C_\mu\in H^k(\fn,P_\Lambda^\cF)$, then Theorem \ref{BBWproj} and Lemma \ref{restrHW} imply that $\mu$ is inside the set $\mu\in w\cdot (\Lambda+\Gamma^+)$,
for some $w\in W(k)$. So in particular the set $\mu-\Gamma^+$ is contained in
$$w\cdot (\Lambda+\Gamma^+)-\Gamma^+=w\circ\Lambda-\widetilde{\Gamma},$$
which is inside one Weyl chamber by assumption. By Definition \ref{defgeneric}, it follows that~$\mu$ is $\Gamma^+$-generic. This means that we can apply the first part of the corollary.
\end{proof}


\section{Relative genericness}
\label{secrelgen}
For every basic classical Lie superalgebra $\fg$ and every Borel subalgebra $\fb$, we define a particular parabolic subalgebra $\fp^{\fb}$. Define $\Pi^a\subset\Delta^+$ as the set of anisotropic positive simple roots. The Levi subalgebra $\fl^{\fb}$ is the subalgebra generated by $\fh$, $\fg_{\alpha}$ and $\fg_{-\alpha}$ for all $\alpha\in \Pi^a$, this is the maximal Levi subalgebra which is of typical type. The maximal parabolic subalgebra of typical type is then defined as $\fp^{\fb}=\fn+\fl^{\fb}$.
\begin{example}{\rm We use the $\Z$-gradations of Subsection \ref{secprel1}.
If $\fg$ is of type I and $\fb$ is the distinguished Borel subalgebra, then $\fp^{\fb}=\fg_0\oplus\fg_1$. If $\fg$ is of type II and $\fb$ is the distinguished Borel subalgebra, then $\fp^{\fb}=\fg_0\oplus\fg_1\oplus\fg_2$. }
\end{example}

The combination of Theorem \ref{thmsimple}, Theorem \ref{thmBBW}$(i)$ and Corollary \ref{degvanish} leads to the following remark.
\begin{remark}\label{remarkmaxtyp}{\rm
The BBW problem for an arbitrary basic classical Lie superalgebra~$\fg$ and Borel subalgebra $\fb$ is solved when the cohomology groups
$$\Gamma_k(G/P^{\fb},L_{\mu}(\fp^{\fb}))$$
are known for $k\le\dim\fu^{\fb}$ and all $\fl^{\fb}$-dominant $\mu\in\cP$.
}
\end{remark}

This remark motives the introduction of a relative notion of genericness. 
\begin{definition}
\label{defrelgen}
Consider a basic classical Lie superalgebra $\fg$, Borel subalgebra $\fb$, $\fl^{\fb}$-dominant weight $\lambda\in\cP$ and some set $S\subset\cP$. We say that $\lambda$ is relatively $S$-generic for $\fb$ if and only if every weight in the set $\lambda-S$, which is $\fl^{\fb}$-dominant, is in the same Weyl chamber as $\lambda$.
\end{definition}

We introduce the notation $W^1_{\fb}:=W^1(\fl_{\oa}^{\fb})$, to stress the dependence on $\fb$. By equation \eqref{MusGS}, we have that $\lambda$ is relatively $\Gamma^+$-generic if and only if $w\cdot\lambda$ is relatively $\Gamma^+$-generic for an arbitrary $w\in W^1_{\fb}$. Similarly we have that $\lambda$ is relatively $\widetilde\Gamma$-generic if and only if $w\circ\lambda$ is relatively $\widetilde\Gamma$-generic for an arbitrary $w\in W^1_{\fb}$.

\begin{example}\label{extrivialrel}{\rm
If $\fg$ is of type I and $\fb$ is the distinguished Borel subalgebra, then every integral dominant weight is relatively $S$-generic for $\fb$ for any set $S$.
}
\end{example}

The notion of relative genericness allows to make part of Theorem \ref{maingen} more general, which is relevant from the point of view of Remark \ref{remarkmaxtyp}.
\begin{proposition}
Consider a basic classical Lie superalgebra $\fg$, Borel subalgebra $\fb$ and a $\fg$-regular integral $\fl^{\fb}$-dominant weight $\mu\in\cP$, which is relatively $\Gamma^+$-generic, then there is exactly one $w\in W^1_{\fb}$, such that $w\cdot\mu$ is integral dominant. Furthermore, we have
$$\Gamma_k(G/B,L_{\mu}(\fb))\cong\Gamma_k(G/P^{\fb},L_{\mu}(\fp^{\fb}))\cong\delta_{k,l(w)}M$$
with $\ch M=\ch K^{(\fb)}_{w\cdot\mu}$.
\end{proposition}
\begin{proof}
We consider $\Gamma_k(G_{\oa}/P_{\oa}^{\fb},\Lambda\fg_{-1}\otimes L_{\mu}(\fp^{\fb}_{\oa}))$ as in Corollary \ref{Penkovresult}. By definition \ref{defrelgen}, the finite dimensional $\fl_{\oa}^{\fb}$-module $\Lambda\fg_{-1}\otimes L_{\mu}(\fl^{\fb}_{\oa})$ decomposes into simple modules with highest weights in the same Weyl chamber as $\mu$. There is a unique $w\in W^{1}_{\fb}$ such that $w\circ\mu\in\cP_{\oa}^+$.

The combination of the two results in Corollary \ref{Penkovresult} therefore implies 
$$\Res \Gamma_k(G,P^{\fb^{\fb}},L_{\mu}(\fp^{\fb}))=\delta_{k,l(w)}\Gamma_l(w)(G_{\oa}/P_{\oa}^{\fb},\Lambda\fg_{-1}\otimes L_{\mu}(\fp^{\fb}_{\oa}))$$

Classical BBW theory and equation \eqref{MusGS} then imply that 
$$\Gamma_l(w)(G_{\oa}/P_{\oa}^{\fb},\Lambda\fg_{-1}\otimes L_{\mu}(\fp^{\fb}_{\oa}))\cong \Gamma_0(G_{\oa}/P_{\oa}^{\fb},\Lambda\fg_{-1}\otimes L_{w\cdot\mu}(\fp^{\fb}_{\oa})),$$
which concludes the proof.
\end{proof}

Combining this result with lemma \ref{Euler} then yields the following corollary.
\begin{corollary}
\label{charKrelgen}
If $\Lambda\in\cP^+$ is relatively $\Gamma^+$-generic, then $\ch K^{(\fb)}_\Lambda=\cE(\Lambda).$
\end{corollary}


\section{Generalised BGG reciprocity}
\label{secBGGrec}

In this section we study the role of the generalised Kac modules $K^{(\fb)}_\Lambda$ in the category $\cF$ for both types of basic classical Lie superalgebras and for arbitrary $\fb$. So either $\cF$ does not have the structure of a highest weight category, or we ignore it. In this setup, a virtual BGG reciprocity was derived by Gruson and Serganova in Section 2 of \cite{Gruson2}. This is summarised in equation \eqref{BGGrecGS} underneath. Our main result is that if, for arbitrary $\fg$ and $\fb$, we work in the relatively generic region (see Definition \ref{defrelgen}), the virtual BGG reciprocity can be expressed as an actual one. So, far away from the walls it seems as if $\cF$ is a highest weight category satisfying the BGG reciprocity. This result includes the BGG reciprocity in \cite{MR1378540}, since relative genericness becomes a trivial condition for type I with distinguished Borel subalgebra, by Example \ref{extrivialrel}.

We introduce a subset of the set of integral dominant weights:
$$\widetilde{\cP}^+=\{\Lambda\in\cP^+ | s_\alpha\cdot\Lambda < \Lambda\mbox{ for every }\alpha\in\Delta^+_{\oa}\}.$$
We summarise the virtual BGG reciprocity of \cite{Gruson2}. There are $a_{\Lambda,\Lambda'}\in\Z$, for all $\Lambda\in\cP^+$ and $\Lambda'\in\widetilde{\cP}^+$, such that
\begin{equation}\label{BGGrecGS}\ch P^{\cF}_\Lambda=\sum_{\Lambda'\in\widetilde{\cP}^+}a_{\Lambda,\Lambda'}\cE(\Lambda')\qquad\mbox{and}\qquad \cE({\Lambda'})=\sum_{\Lambda\in\cP^+}a_{\Lambda,\Lambda'}\ch L_{\Lambda}.\end{equation}

\begin{theorem}
\label{BGGrecGSC}
Consider $\fg$ a basic classical Lie superalgebra and $\fb$ an arbitrary Borel subalgebra. Assume that $\Lambda\in\cP$ is relatively $\widetilde{\Gamma}$-generic for $\fb$. Then $P_\Lambda^{\cF}$ has a filtration by generalised Kac modules $K^{(\fb)}_{\Lambda'}$, with $\Lambda'\in\widetilde{\cP}^+$, satisfying
$$(P_\Lambda^{\cF}:K^{(\fb)}_{\Lambda'})=a_{\Lambda,\Lambda'}=[K_{\Lambda'}^{(\fb)}: L_\Lambda].$$
\end{theorem}

The remainder of this section is devoted to the proof of this result. Theorem \ref{BBWproj} leads to the following conclusion on the connection between projective and generalised Kac modules.
\begin{corollary}
\label{JHKacI}
The Jordan-H\"older decomposition of the generalised Kac module $K^{(\fb)}_{\Lambda'}$ for $\Lambda'\in\cP^+$ satisfies
\[[K^{(\fb)}_{\Lambda'}:L_{\Lambda}]=\Hom_{\fh}(\C_{\Lambda'}, H^0(\fn, P^\cF_\Lambda ))\]
for any $\Lambda\in\cP^+$. In particular 
\[\Hom_{\fh}(\C_{\Lambda}, H^0(\fn, P^\cF_\Lambda ))=1\quad \mbox{and}\quad \Hom_{\fh}(\C_{\Lambda'}, H^0(\fn, P^\cF_\Lambda ))=0\mbox{ if }\Lambda'\not\ge\Lambda. \]
\end{corollary}
\begin{proposition}
\label{BGGrec}
Consider $\Lambda\in\cP^+$, if $P^\cF_\Lambda $ has a filtration by generalised Kac modules, the multiplicities are given by
\[(P^\cF_\Lambda :K^{(\fb)}_\lambda)=[K^{(\fb)}_\lambda:L_\Lambda]\quad\mbox{for any $\lambda\in\cP^+$.}\]
\end{proposition}
\begin{proof}
Since the projective modules in $\cF$ are their own twisted duals, the relation 
\[ \Hom_{\fh}(\C_{\Lambda'}, H_0(\ofn, P^\cF_\Lambda ))=\Hom_{\fh}(\C_{\Lambda'}, H^0(\fn, P^\cF_\Lambda ))=[K^{(\fb)}_{\Lambda'}:L_{\Lambda}]\]
holds as a consequence of Corollary \ref{JHKacI}.

Now consider an arbitrary module $M$, with a finite filtration with quotients the generalised Kac modules $\{K^{(\fb)}(\kappa)|\kappa\in S\}$ for some set $S$ with multiplicities. Since $\Hom_{\ofn}(-,\C)$ is a left-exact contravariant functor, it is clear that $H_0(\ofn,M)$ is an $\fh$-submodule of $\oplus_{\kappa\in S}\C_\kappa$. Since the generalised Kac modules correspond to the maximal finite dimensional highest weight modules, it also follows that  $\oplus_{\kappa\in S}\C_\kappa$ must be a submodule of $H^k(\ofn,M)$. This implies that we have the relation $$(P_\Lambda^{\cF} :K^{(\fb)}_\lambda)=\Hom_{\fh}(\C_{\lambda}, H_0(\ofn, P^\cF_\Lambda )),$$
which concludes the proof.
\end{proof}
 
\begin{lemma}
\label{helpconnGS}
Assume that for a fixed $\Lambda\in\cP^+$, 
\begin{itemize}
\item $a_{\Lambda,\Lambda'}\not=0$ implies that $\ch K^{(\fb)}_{\Lambda'}=\cE(\Lambda')$, 
\item $\Hom_{\fh}(\C_{\kappa},H_{0}(\ofn,P_\Lambda^{\cF}))\not=0$ implies $\kappa\in \widetilde{\cP}^+$.
\end{itemize}
Then $P_\Lambda^{\cF}$ has a filtration by generalised Kac modules and
\[(P^\cF_\Lambda :K^{(\fb)}_{\Lambda'})=a_{\Lambda,\Lambda'}=[K^{(\fb)}_{\Lambda'}:L_\Lambda]\quad\mbox{for all $\Lambda'\in\widetilde{\cP}^+$.}\]
\end{lemma}
\begin{proof}
Take $\Lambda'\in\widetilde{\cP}^+$ such that $a_{\Lambda,\Lambda'}\not=0$. Then by assumption, $$a_{\kappa,\Lambda'}=[K_{\Lambda'}^{(\fb)}: L_{\kappa}]\ge 0,$$ for any $\kappa\in\cP^+$. Equation \eqref{BGGrecGS} and Corollary \ref{JHKacI} therefore imply that $$\ch P^{\cF}_{\Lambda}=\sum_{\Lambda'\in\widetilde{\cP}^+}\Hom_{\fh}(\C_{\Lambda'},H_0(\ofn,P^{\cF}_{\Lambda}))\ch K_{\Lambda'}^{(\fb)}.$$
In general, $P^{\cF}_\Lambda$ has a filtration by certain quotients of generalised Kac modules, where the highest weights are exactly given by the set (with multiplicities) $H_0(\ofn,P^{\cF}_{\Lambda})$. By assumption this set is contained in $\widetilde{\cP}^+$. The only possibility for $P^{\cF}_\Lambda$ to have the character as written above is therefore if all these quotients are isomorphic to the generalised Kac modules. The result thus follows from Proposition \ref{BGGrec}.
\end{proof}

\begin{lemma}
\label{genP+}
If a weight $\kappa\in\cP^+$ is relatively $\Gamma^+$-generic, it satisfies $\kappa\in\widetilde{\cP}^+$.
\end{lemma}
\begin{proof}
By Corollary \ref{charKrelgen} we have $\cE(\kappa)=\ch K^{(\fb)}_\kappa$. If $s_\alpha\cdot\kappa=\kappa$ for some $\alpha\in\Delta^+_{\oa}$, we obtain the contradiction $\cE(\kappa)=0$, so there are no multiplicities in the orbit $\{w\cdot\kappa\,|\,w\in W\}$. The highest weight in this orbit has to appear with non-zero multiplicity in $\cE(\kappa)$ by equation~\eqref{formEuler}. This implies that the highest weight in that orbit is $\kappa$, so $s_\alpha\cdot\kappa<\kappa$.
\end{proof}

\begin{proof}[Proof of Theorem \ref{BGGrecGSC}]
It suffices to prove that the conditions in Lemma \ref{helpconnGS} are satisfied if $\Lambda$ is relatively $\widetilde{\Gamma}$-generic. 

Assume that $a_{\Lambda,\Lambda'}\not=0$ for some $\Lambda'\in\widetilde{\cP}^+$, the combination of equation \eqref{BGGrecGS} and Lemma \ref{Euler} implies that 
$$[\Gamma_i(G/P^\fb,L_{\Lambda'}(\fp^\fb)):L_\Lambda]\not=0$$
for some $i$. By Lemma\ref{restrHW}, there is an $u\in W^1_{\fb}$ such that $\Lambda'\in u\circ\Lambda+\Gamma^+$. Since $\Lambda$, and thus $u\circ\Lambda$, is relatively $\widetilde{\Gamma}$-generic, this implies that $u=1$ and $\Lambda'$ is $\Gamma^+$-generic. Therefore Corollary \ref{charKrelgen} yields $\ch K^{(\fb)}_{\Lambda'}=\cE(\Lambda')$.

Now assume that $\Hom_{\fh}(\C_{\kappa},H_{0}(\ofn,P_\Lambda^{\cF}))\not=0$. Theorem \ref{BBWproj} implies that we have $[K_{\kappa}^{(\fb)}:L_{\Lambda}]\not=0$. By Corollary \ref{restrHW}, this implies $\Lambda\in \kappa-\Gamma^+$ and in particular, $\kappa$ is relatively $\Gamma^+$-generic. Lemma \ref{genP+} then implies that $\kappa\in\widetilde{\cP}^+$.
\end{proof}


\section{BBW theory for $\mathfrak{osp}(m|2)$, $D(2,1;\alpha)$, $F(4)$ and $G(3)$}
\label{seccomplres}

We extend the result on BBW theory for $\mathfrak{osp}(3|2)$, $D(2,1;\alpha)$, $F(4)$ and $G(3)$, with distinguished root system, of Germoni and Martirosyan in \cite{Germoni, Lilit}, to include weights which are not necessarily dominant. This solves BBW theory for these algebras (with distinguished Borel subalgebra) completely. We also solve BBW theory for $\mathfrak{osp}(m|2)$ by applying the results of Su and Zhang on Kac modules and generalised Verma modules in \cite{SuZha}. We always assume that $m\ge 3$, since the other cases have already been addressed in Subsection \ref{subtyptypeI}. The remaining basic classical Lie superalgebras for which BBW theory, for the distinguished system of positive roots, is not known, are therefore $\mathfrak{osp}(m|2n)$, with $n\ge 2$ and $m\ge 3$. 

All the Lie subalgebras in this section are of type II, therefore they satisfy the $\Z$-gradation $\fg=\fg_{-2}\oplus\fg_{-1}\oplus\fg_{0}\oplus\fg_{1}\oplus\fg_{2}$. The distinguished Borel subalgebra satisfies $\fb\subset\fg_{0}\oplus\fg_1\oplus\fg_2$. Furthermore, in each case we have $\dim\fg_2=1$ and $\fg_{\oa}\cong\mathfrak{sl}(2)+\mathfrak{g}_{0}$. This implies that the Weyl group satisfies $W\cong\Z_2\times W(\fg_0:\fh)$, where we denote the non-trivial element of $\Z_2$ by $s$.

\begin{theorem}
\label{excep}
Consider $\fg$ one of the basic classical Lie superalgebras in the list $\{\mathfrak{osp}(m|2),D(2,1;\alpha), F(4),G(3)\}$ and $\fb$ the distinguished Borel subalgebra, see \cite{MR0519631}. For each $\mu\in\cP$ and $p\in\N$, there is at most one $w\in W$ of length $p$ such that $w\cdot\mu\in\cP^+$. The cohomology groups of BBW theory are described by
\[\Gamma_p(G/B,L_\mu(\fb)=\begin{cases}K_{w\cdot\mu}&\mbox{if $w\in W(p)$ exists and }s\not\uparrow w\\K_{w\cdot\mu}^\vee&\mbox{if $w\in W(p)$ exists and }s\uparrow w\\0&\mbox{otherwise.}\end{cases}\]
\end{theorem}
Before proving this statement, we note the following corollary on Kostant cohomology of projective modules.
\begin{corollary}
\label{KostProj2}
Consider $\fg\in\{\mathfrak{osp}(m|2),D(2,1;\alpha), F(4),G(3)\}$ and $\fb$ the distinguished Borel subalgebra. For each $\Lambda\in\cP^+$ we have
$$\ch H^k(\fn,P_\Lambda^{\cF})=\bigoplus_{w\in W(k)}w\cdot\ch H^0(\fn, P_\Lambda^{\cF}).$$
\end{corollary}
\begin{proof}
Theorem \ref{excep} implies that for any integral weight $\mu$ and integral dominant weight $\Lambda$ we have
$$[\Gamma_p(G/B,L_\mu(\fb)):L_\Lambda]=\begin{cases}[\Gamma_0(G/B,L_{w\cdot \mu}(\fb)):L_\Lambda] &\mbox{if there is a }w\in W(p)\\&\mbox{such that }w\cdot\mu\in\cP^+\\0&\mbox{otherwise.}\end{cases}$$
The result therefore follows from Theorem \ref{BBWproj}(ii).
\end{proof}

\begin{lemma}
\label{lemmaDFG}
Consider $\fg\in\{\mathfrak{osp}(m|2),D(2,1;\alpha), F(4),G(3)\}$ and $\fp=\fg_0\oplus\fg_1\oplus\fg_2$. The $\fg$-module $\Gamma_1(G/P,L_{s\cdot\Lambda}(\fp))$ contains no highest weight vectors lower than $\Lambda$.
\end{lemma}
\begin{proof}
Denote the positive root in $\fg_2$ by $2\delta$. We consider a $k\in\N$  such that $\Lambda+k\delta$ is typical. Theorem \ref{BBWtypical} thus implies that we have
$$ \Gamma_i(G/P, L_{s\cdot(\Lambda+k\delta)})=\delta_{i1} L_{\Lambda+k\delta}.$$ 
Since $s\cdot (\Lambda+k\delta)=s\cdot\Lambda -k\delta$, we have a short exact sequence of $\fp$-modules
$$L_{s\cdot \Lambda}(\fp)\hookrightarrow L_{s\cdot(\Lambda+k\delta)}(\fp)\otimes L_{k\delta} \tto N,$$
for some $\fp$-module $N$. We apply the right exact functor $\Gamma(G/P,-)$ to this short exact sequence, using Corollary \ref{degvanish} and Corollary \ref{tensorfin}, which yields an exact sequence
$$0\to\Gamma_1(G/P,L_{s\cdot\Lambda}(\fp))\to  L_{\Lambda+k\delta}\otimes L_{k\delta}.$$
The Weyl group invariance of $\ch L_{k\delta}$, shows that $-k\delta$ is the lowest weight appearing in $L_{k\delta}$, which implies that the lowest possible weight of a non-zero highest weight vector in $L_{\Lambda+k\delta}\otimes L_{k\delta}$ is $\Lambda$.
\end{proof}

\begin{proof}[Proof of Theorem \ref{excep}]
We prove that $\Gamma_1(G/B,L_{s\cdot\Lambda})=K_{\Lambda}^{\vee}$. By Theorem \ref{thmBBW}(i) we can replace the Borel subalgebra by $\fp=\fg_0\oplus\fg_1\oplus\fg_2$. The remainder of the theorem then follows from Corollary \ref{degvanish} and Theorem \ref{thmsimple}, see also Remark \ref{remarkmaxtyp}.

If $\Lambda$ is typical, the result follows from Theorem \ref{BBWtypical}. Is $\Lambda$ is atypical, the results in \cite{Germoni, Lilit, SuZha} imply that there are three possibilities for $K_{\Lambda}$
\begin{itemize}
\item $K_\Lambda$ is of length 2, $s\cdot\Lambda\not\in\cP^+$;
\item $K_\Lambda$ is of length 3, $s\cdot\Lambda\not\in\cP^+$, there exists no extension between the two simple subquotients in the maximal submodule of $K_\Lambda$
\item $K(\Lambda)=L(\Lambda)$, $s\cdot\Lambda\in \cP^+$ (sometimes actually $s\cdot\Lambda=\Lambda$)
\end{itemize}

For the first two cases, the cohomologies $\Gamma_i(G/P,L_{s\cdot\Lambda}(\fp))$ are clearly contained in the first degree. The only possibility allowed by Lemma \ref{Euler} and Lemma \ref{lemmaDFG} is $\Gamma_i(G/P,L_{s\cdot\Lambda}(\fp))=K_\Lambda^{\vee}$. In the third case, we have the identity
$$\cE(\Lambda)=\ch L(\Lambda)-\ch L(s\cdot\Lambda),$$ see \cite{Germoni, Lilit, SuZha}, while $K(\Lambda)\cong L(\Lambda)$ and $K(s\cdot\Lambda)\cong L(s\cdot\Lambda)$. The result then follows from Lemma \ref{Euler}.
\end{proof}

\begin{remark}
We could have avoided using Lemma \ref{lemmaDFG}, by using the Serre duality derived by Penkov in the sheaf-theoretical approach to BBW theory in \cite{MR0957752}:
\begin{equation}\label{eqSerreD}\Gamma_k(G/B,L_{\lambda}(\fb))\cong \Gamma_{d-k}(G/B, L_{-\lambda-2\rho}(\fb))^\ast,\end{equation}
with $d=\fn_{\oa}^\ast$. The combinantion of this with the results in \cite{Germoni, Lilit, SuZha} also leads to Theorem \ref{excep}. Equation \eqref{eqSerreD} also demonstrates that cohomology of BBW theory will not always lead to highest weight modules. For type I the Kac modules satisfy $K_\lambda^\ast=K_{-w_0(\lambda)+2\rho_1}$, but duals of arbitrary generalised Kac modules will not always be generalised Kac modules.

\end{remark}


\section{Homological algebra and projective modules for $\mathfrak{osp}(m|2)$}
\label{sec32}
In this section we study homological algebra for the Lie superalgebra $\mathfrak{osp}(m|2)$. Since we will not use the BGG category $\cO$ here, we denote the indecomposable projective covers in $\cF$ simply by $P_\Lambda$.

First we repeat some results of Su and Zhang in \cite{SuZha} and Gruson and Serganova in \cite{MR2734963}. The defect of $\mathfrak{osp}(m|2)$ is $1$, so every atypical central character is singly atypical. By \cite{MR2734963}, we therefore know that every atypical block in $\cF$ for $\mathfrak{osp}(m|2)$ is equivalent to an atypical block in $\mathfrak{osp}(3|2)$ if $m$ is odd, or $\mathfrak{osp}(4|2)$ or $\mathfrak{osp}(2|2)\cong \mathfrak{sl}(2|1)$ if $m$ is even. The quiver diagrams of these categories can therefore be obtained from the ones in \cite{Germoni}.
\begin{lemma}
\label{quivers}
(i) If $\fg=\mathfrak{osp}(2d+1|2)$, the quiver diagram of $\cF_\chi$, for $\chi$ an atypical central character is equal to the Dynkin diagram of type $D_{\infty}$.

(ii) If $\fg=\mathfrak{osp}(2d|2)$, the quiver diagram of $\cF_\chi$, for $\chi$ an atypical central character is either equal to the Dynkin diagram of type $D_{\infty}$, or of type $A^\infty_{\infty}$.
\end{lemma}
These results also follow from interpreting Theorem 4.2 in \cite{SuZha}. We follow the notation for the weights introduced in Definition 2.9 in \cite{SuZha} for the remainder of this section. The structure of the Kac modules can be obtained from Theorem 4.2 in \cite{SuZha}, while the characters of finite dimensional modules are described in Proposition 4.6 in \cite{SuZha}. As in Section \ref{seccomplres}, we denote by $s\in W$ the simple reflection for the root which is not simple in $\Delta^+$. 
\begin{lemma}
\label{structKacospm2}
Consider $\fg=\mathfrak{osp}(m|2)$. If the quiver diagram of the block $\cF_\chi$ is of type $D_\infty$, the integral dominant weights corresponding to $\cF_{\chi}$ are given by a set $\{\lambda^{(0)},\lambda^{(1)},\lambda^{(2)},\cdots\}$ and the Kac modules have the following description:
$$L_{\lambda^{(k-1)}}\hookrightarrow K_{\lambda^{(k)}}\tto L_{\lambda^{(k)}}\quad\mbox{if } k\ge 2,$$
\vspace{-7mm}
$$L_{\lambda^{(0)}}\oplus L_{\lambda^{(1)}}\hookrightarrow K_{\lambda^{(2)}}\tto L_{\lambda^{(2)}},\quad$$
\vspace{-6mm}
$$K_{\lambda^{(1)}}\cong L_{\lambda^{(1)}}\quad\mbox{and}\quad K_{\lambda^{(0)}}\cong L_{\lambda^{(0)}}.$$
Furthermore, we have $\ch K_{\lambda^{(k)}}=\cE(\lambda^{(k)})$ if $k\ge 2$ and $\cE(\lambda^{(1)})=\ch L_{\lambda^{(1)}}-\ch L_{\lambda^{(0)}}$.

If the quiver diagram of the block $\cF_\chi$ is given by $A^\infty_\infty$, the integral dominant weights are given by a set $\{\cdots,\lambda^{(2)}_-,\lambda^{(1)}_-,\lambda^{(0)}_-=\lambda^{(0)}_+,\lambda^{(1)}_+,\lambda^{(2)}_+,\cdots\}$ and the Kac modules have the following description:
$$L_{\lambda^{(k-1)}_{\pm}}\hookrightarrow K_{\lambda^{(k)}_{\pm}}\tto L_{\lambda^{(k)}_{\pm}}\quad\mbox{if } k\ge 1, \quad\mbox{ and }K_{\lambda^{(0)}_+}\cong L_{\lambda^{(0)}_+}.$$
Furthermore, we have $\ch K_{\lambda^{(k)}_{\pm}}=\cE(\lambda^{(k)}_{\pm})$ if $k\ge 1$.
\end{lemma}

We note that the weights $\lambda^{(k)}$ and $\lambda^{(k)}_{\pm}$ are defined in \cite{SuZha} for $k\in\Z$, but are not dominant if $k\le 0$. The principle $s\cdot \lambda^{(k)}=\lambda^{(1-k)}$ and $s\cdot\lambda_{\pm}^{(k)}=\lambda_{\pm}^{(-k)}$ holds.

In the remainder of this section we will state results on an arbitrary block by use of the abstract notation $\lambda^{(k)},\lambda^{(k)}_+,\lambda^{(k)}_-$. From notation it is therefore clear which type of block is considered.
\subsection{Projective modules}

\begin{proposition}
\label{projospm2}
For $k\ge 2$, the projective modules $P_{\lambda^{(k)}}$ satisfy
$$K_{\lambda^{(k+1)}}\hookrightarrow P_{\lambda^{(k)}}\tto K_{\lambda^{(k)}},$$
and the length of the radical layer structure is three. For $k\in\{1,2\}$ the radical layer structure of $P_{\lambda^{(k)}}$ is
$$L_{\lambda^{(k)}}\qquad\quad L_{\lambda^{(2)}}\qquad\quad L_{\lambda^{(k)}}.$$

For $k\ge 1$, the projective modules $P_{\lambda_{\pm}^{(k)}}$ satisfy
$$K_{\lambda^{(k+1)}_{\pm}}\hookrightarrow P_{\lambda_{\pm}^{(k)}}\tto K_{\lambda^{(k)}_\pm},$$
and the length of the radical layer is three. The radical layer structure of $P_{\lambda^{(0)}}$ is
$$L_{\lambda_+^{(0)}}\qquad\quad L_{\lambda^{(1)}_+}\oplus L_{\lambda^{(1)}_-}\qquad\quad L_{\lambda_+^{(0)}}.$$
\end{proposition}
\begin{proof}
The Jordan-H\"older decomposition series of the projective modules follows from comparing the BGG reciprocity in Theorem 1 of \cite{Gruson2} with the characters of the Kac modules in Lemma \ref{structKacospm2}. Alternatively they follow quickly from the Euler-Poincar\'e principle and the subsequent Lemma \ref{KostProj}.

Since $P_\Lambda$ is also the indecomposable injective envelope of $L_\Lambda$, we have $\Soc P_\Lambda= L_\Lambda$. The fact that $\Rad (P_\Lambda/L_\Lambda)$ is semisimple follows from its decomposition series and the quiver diagrams in Lemma \ref{quivers}.

The filtration by Kac modules follows from the fact that $P_\Lambda$ projects onto $K_\Lambda$ and the fact the kernel of this morphism is a module with simple socle $L_\Lambda$.
\end{proof}

\begin{remark}
This result provides examples of Theorem \ref{BGGrecGSC}. It also shows that the projective modules exceptional weights $\lambda^{(1)},\lambda^{(0)},\lambda^{(0)}_+$, which are not relatively $\widetilde{\Gamma}$-generic, do not have a filtration by Kac modules.
\end{remark}

\subsection{Bernstein-Gelfand-Gelfand resolutions}
We call a resolution by (generalised) Verma modules a BGG resolution and a resolution by modules with a filtration by generalised Verma modules a weak BBG resolution, see \cite{lepowsky}.
According to Section 8 in \cite{BGG} and Remark \ref{MorGorelik} we have the following conclusions for basic classical Lie superalgebras with distinguished Borel subalgebra and a parabolic subalgebra such that the Levi subalgebra is even.\begin{itemize}
\item If $\fg$ is of type I, then Kac modules and typical simple modules have a finite BGG resolution.
\item If $\fg$ is of type II, then typical simple modules have a finite BGG resolution.
\end{itemize}
In this section we look at such resolutions for Kac modules and Lie superalgebras of type II. We obtain the following conclusions:
\begin{theorem}
\label{sumBGG}
Consider $\fg$ a basic classical Lie superalgebra of type II, with parabolic subalgebra $\fp=\fg_0\oplus\fg_1\oplus\fg_2$.
\begin{itemize}
\item An atypical Kac module never has a finite (weak) BGG resolution.
\item If $\fg=\mathfrak{osp}(m|2)$, then an atypical Kac module either has a infinite BGG resolution or a finite resolution by twisted generalised Verma modules.
\end{itemize}
\end{theorem}

We use the notation as in equation \eqref{KacII} and the beginning of this section.
\begin{theorem}
\label{BGG32}
Consider $\fg=\mathfrak{osp}(m|2)$. For $\lambda\in\cP^+$ typical and parabolic subalgebra $\fp=\fg_0+\fg_1+\fg_2$, the irreducible module $L_\lambda=K_{\lambda}$ has a BGG resolution
\[0\to \overline{K}_{s\cdot\lambda}\to\overline{K}_{\lambda}\to L_{\lambda}\to 0.\]

For $k\ge 2$, the Kac module $K_{\lambda^{(k)}}$ has a resolution of the form
$$0\to N_{\lambda^{(k)}}\to \overline{K}_{\lambda^{(k)}}\to K_{\lambda^{(k)}}\to 0,$$
where $\ch N_{\lambda^{(k)}}=\ch \overline{K}_{\lambda^{(1-k)}}$, but $H_0\left(\ofn, N_{\lambda^{(k)}}\right)=\C_{\lambda^{(1-k)}}\oplus \C_{\lambda^{(-k)}}$. For $k\in\{0,1\}$, the module $K_{\lambda^{(k)}}\cong L_{\lambda^{(k)}}$ has a BGG resolution of the form
\begin{eqnarray*}
\cdots \to \overline{K}_{\lambda^{(-j)}}\to\cdots \to\overline{K}_{\lambda^{(-2)}}\to \overline{K}_{\lambda^{(-1)}}\to \overline{K}_{\lambda^{(k)}}\to L_{\lambda^{(k)}}\to 0
\end{eqnarray*}

For $k\ge 1$, the Kac module $K_{\lambda^{(k)}_{\pm}}$ has a resolution of the form
$$0\to N_{\lambda^{(k)}_{\pm}}\to \overline{K}_{\lambda^{(k)}_{\pm}}\to K_{\lambda^{(k)}_{\pm}}\to 0,$$
where $\ch N_{\lambda^{(k)}_{\pm}}=\ch \overline{K}_{\lambda^{(-k)}_{\pm}}$, but $H_0\left(\ofn, N_{\lambda^{(k)}_{\pm}}\right)=\C_{\lambda^{(-k)}_{\pm}}\oplus \C_{\lambda^{(-k-1)}_{\pm}}$. The module $K_{\lambda^{(0)}_+}\cong L_{\lambda^{(0)}_+}$ has a BGG resolution of the form
\begin{eqnarray*}
&&\cdots \to  \overline{K}_{\lambda^{(-j)}_+}\oplus  \overline{K}_{\lambda^{(-j)}_-}\to\cdots \to \overline{K}_{\lambda^{(-2)}_+}\oplus  \overline{K}_{\lambda^{(-2)}_-}\\
&&\to \overline{K}_{\lambda^{(-1)}_+}\oplus  \overline{K}_{\lambda^{(-1)}_-}\to \overline{K}_{\lambda^{(0)}_+}\to L_{\lambda^{(0)}_+}\to 0.
\end{eqnarray*}
\end{theorem}
\begin{proof}
First we consider $\lambda$ typical, according to equation \eqref{KacII} there is a short exact sequence
$$N_\lambda\hookrightarrow \overline{K}_\lambda\tto L_{\Lambda}$$
with $\ch N_\lambda= \ch\overline{K}_{s\cdot\lambda}$. Corollary \ref{Kostanttypical} implies $H_0(\ofn, N_\lambda)=L^0_{s\cdot\Lambda}$, so $N_\lambda= \overline{K}_{s\cdot\lambda}$.

The atypical cases follow immediately from Theorem 4.2 in \cite{SuZha}.
\end{proof}

\begin{proof}[Proof of Theorem \ref{sumBGG}]
If a finite dimensional $\fg$-module, restricted as an $\ofn$-module, had a finite resolution by free $\ofn$-modules, it would have projective dimension zero, as a module for the one dimensional Lie algebra generated by any self-commuting element in $\ofn$. For type II, this property immediately implies that the module is projective in $\cF$, see \cite{Duflo}. This proves the first statement. For the case of $\mathfrak{osp}(m|2)$, this also follows from the subsequent Theorem \ref{KostantKac}.

The second statement follows from Theorem \ref{BGG32}.
\end{proof}

\begin{remark}
For basic classical Lie superalgebras of type II, the Kac modules are not parabolically induced, so resolutions in terms of them are not BGG resolutions. However, Lemma \ref{structKacospm2} implies immediately that each simple module for $\mathfrak{osp}(m|2)$ has a finite resolution in terms of Kac modules.
\end{remark}

\subsection{Kostant cohomology}
The main theorem of this subsection is the algebra homology of $\ofn$ with values in the Kac modules of $\mathfrak{osp}(m|2)$. 

According to Corollary \ref{Kostanttypical} we only need to focus on atypical weights. By Lemma~\ref{KostantBott} and Remark \ref{Kostevenrefl} it suffices to compute the homology of $\ofu=\fg_{-1}\oplus\fg_{-2}$, which is what we do in the following Theorem.
\begin{theorem}
\label{KostantKac}
For every $\lambda\in\cP^+$, $H_0(\ofn,K_\lambda)=\C_\lambda$, and for $j>0$
$$H_j(\ofu, L_{\lambda^{(k)}})=L_{\lambda^{(-j)}}^0\quad\mbox{for}\quad k\in\{0,1\},\quad H_j(\ofu,L_{\lambda_+^{(0)}})=L_{\lambda_+^{(-j)}}^0\oplus L^0_{\lambda_-^{(-j)}},$$
$$H_j(\ofu, K_{\lambda^{(k)}})=L^0_{\lambda^{(2-j-k)}}\oplus L^0_{\lambda^{(1-j-k)}}\quad\mbox{and}\quad H_j(\ofu, K_{\lambda^{(k)}_{\pm}})=L^0_{\lambda^{(1-j-k)}_{\pm}}\oplus L^0_{\lambda^{(-j-k)}_{\pm}}$$
for respectively $k\ge 2$ and $k\ge 1$.
\end{theorem}

In order to prove this we need the Kostant cohomology for projective modules.
\begin{lemma}
\label{KostProj}
Kostant homology of projective modules for $\mathfrak{osp}(m|2)$ is described by
\[H^j(\fn,P_{\lambda^{(k)}} )=\bigoplus_{w\in W(j)}\left(\C_{w\cdot\lambda^{(k)}}\oplus \C_{w\cdot\lambda^{(k+1)}}\right) \quad \mbox{for}\quad k>0;\]
\[H^j(\fn,P_{\lambda^{(0)}} )=\bigoplus_{w\in W(j)}\left(\C_{w\cdot\lambda^{(0)}}\oplus \C_{w\cdot\lambda^{(2)}}\right) ;\]
\[H^j(\fn,P_{\lambda^{(k)}_{\pm} })=\bigoplus_{w\in W(j)}\left(\C_{w\cdot\lambda^{(k)}_{\pm}}\oplus \C_{w\cdot\lambda^{(k+1)}_{\pm}}\right) \quad \mbox{for}\quad k>0;\]
\[H^j(\fn,P_{\lambda^{(0)}_+} )=\bigoplus_{w\in W(j)}\left(\C_{w\cdot\lambda^{(0)}_+}\oplus \C_{w\cdot\lambda^{(1)}_+} \oplus \C_{w\cdot \lambda^{(1)}_-}\right).\]
\end{lemma}
\begin{proof}
This follows from the combination of Corollary \ref{JHKacI}, Lemma \ref{structKacospm2} and Corollary \ref{KostProj2}.
\end{proof}

\begin{proof}[Proof of Theorem \ref{KostantKac}]
 For the cases where the Kac module is simple, the result follows immediately from Theorem \ref{BGG32}. Now we prove the result for $\lambda^{(k)}$ with $k\ge 2$.

Theorem \ref{BGG32} implies \[H_1(\ofu,K_{\lambda^{(k)}})=L^0_{\lambda^{(1-k)}}\oplus L^0_{\lambda^{(-k)}}\quad \mbox{for all}\,\, k\ge 2.\]
We make the identification $H_k(\ofu,V)=H^k(\fu,V^\vee)$ (see e.g. Lemma 6.22 in \cite{MR3012224} or Remark 4.1 in \cite{BGG}) to use the result in Lemma \ref{KostProj}. Applying $\Hom_{\ofu}(-,\C)^\ast=H_0(\ofu,-)$ on the short exact sequence in Proposition \ref{projospm2} yields a long exact sequence of the form
\begin{eqnarray*}
&&0\to  L^0_{\lambda^{(k)}} \to L^0_{\lambda^{(k+1)}}\oplus L^0_{\lambda^{(k)}}\to L^0_{\lambda^{(k+1)}}\to H_1(\ofu,K_{\lambda^{(k)}}) \to L^0_{\lambda^{(-k)}}\oplus L^0_{\lambda^{(1-k)}}\to\\
&&  H_1(\ofu,K_{\lambda^{(k+1)}})\to H_2(\ofu,K_{\lambda^{(k)}})  \to 0\to  H_2(\ofu,K_{\lambda^{(k+1)}})\to H_3(\ofu,K_{\lambda^{(k)}})  \to\cdots.
\end{eqnarray*}
This implies $H_j(\ofu,K_{\lambda^{(k+1)}})\cong H_{j+1}(\ofu,K_{\lambda^{(k)}})$ for $k\ge 2$ and $j>0$. The proof for $K_{\lambda^{(k)}_{\pm}}$ with $k>0$ follows identically.
\end{proof}



\section{Kostant cohomology of projective modules in $\cF$}
\label{HknP}

We have a unifying formula for Kostant cohomology of projective modules in $\cF$, which holds for $\mathfrak{sl}(m|n)$, $\mathfrak{osp}(1|2n)$, $\mathfrak{osp}(2|2n)$, $\mathfrak{osp}(m|2)$, $D(2,1;\alpha)$, $F(4)$ and $G(3)$ with distinguished Borel subalgebra, which also holds for arbitrary basic classical Lie superalgebras and arbitrary Borel subalgebras in the generic and typical regions. The results of Corollary \ref{KosProjI}, Corollary \ref{Kostanttypical}, Corollary \ref{KoProGen}$(ii)$, Corollary \ref{KostProj2} can thus be summarised as.

\begin{proposition}
Consider $\fg$ a basic classical Lie superalgbra, $\fb=\fh\oplus\fn$ a Borel subalgebra and $\Lambda\in\fh^\ast$ an integral dominant weight. If one of the conditions
\begin{itemize}
\item $\fg$ is of type I, or equal to $\mathfrak{osp}(m|2)$, $D(2,1;\alpha)$, $G(3)$ or $F(4)$, with $\fb$ the distinguished Borel subalgebra;
\item  $\Lambda$ is typical or $\widetilde{\Gamma}$-generic;
\end{itemize}
is satisfied, we have
$$\ch H^k(\fn,P_\Lambda^{\cF})=\bigoplus_{w\in W} w\cdot\ch H^0(\fn,P_\Lambda^{\cF}).$$
\end{proposition}

We prove that this result does not extend to basic classical Lie superalgebras with defect greater than 1.
\begin{proposition}
Consider $\fg=\mathfrak{osp}(m|2n)$, with $n>1$ and $m>3$, there exists a $\Lambda\in\cP^+$ such that we have the inequality
\[\ch H^1(\fn,P_\Lambda )\not=\bigoplus_{w\in W(1)}w\cdot \ch H^0(\fn,P_\Lambda ).\]
\end{proposition}
\begin{proof}
Consider constant $k,l\in\N$ such that $l\le m-2$ and $m-2-l< k+1< l$ hold (e.g. $k=0$ and $l=m-2$) and the weight
\[\lambda=k\delta_1+l\delta_2.\]
Then both $\Lambda_1=r_{\delta_1-\delta_2}\cdot \lambda$ and $\Lambda_2=r_{2\delta_2}\cdot\lambda$ are $\fg$-integral dominant (with $\Lambda_2 <\Lambda_1$). Theorem \ref{thmsimple} implies that
\[\Gamma_1(G/B,L_{\lambda}(\fb))=K_{\Lambda_1}\]
holds. We denote the multiplicity $[K_{\Lambda_1}:L_{\Lambda_2}]$ by $p$. Theorem \ref{BBWproj} implies that $\C_{\lambda}$ appears $p$ times in $H^1(\fn,P_{\Lambda_2})$ and that $\C_{\Lambda_1}$ appears $p$ times in $H^0(\fn,P_{\Lambda_2})$. However since $\C_{\Lambda_2}$ also appears in $H^0(\fn,P_{\Lambda_2})$ the equality
\[\ch H^1(\fn,P_{\Lambda_2})=\bigoplus_{w\in W(1)}w\cdot \ch H^0(\fn,P_{\Lambda_2})\]
would lead to a contradiction.
\end{proof}

\appendix

\section{A few results on twisting functors}
\label{aptwist}

In the following, $\alpha$ is always a root which is simple in $\Delta_{\oa}^+$. We say that $M\in\cO$ is $\alpha$-free (respectively $\alpha$-finite) if for a non-zero $Y\in(\fg_{\oa})_{-\alpha}$ the action of $Y$ is injective (respectively locally finite) on $M$. A simple module is either $\alpha$-finite or $\alpha$-free. We introduce the partial Zuckerman functor $S_\alpha$ on $\cO$, which maps a module to its maximal $\alpha$-finite submodule and the partial Bernstein functor $\Gamma_\alpha$, which maps a module to its maximal $\alpha$-finite quotient. We denoted the derived category of $\cO$ by $\cD(\cO)$, with $\cD^+(\cO)$ and $\cD^-(\cO)$ respectively the bounded-below and bounded-above derived categories.

In the following lemma, we recall properties of the twisting functors $T_\alpha$ and their adjoint $G_\alpha$ from Subsection \ref{prelsec4}, which can be found in Lemma 5.4 and Propositions~5.10 and 5.11 of \cite{CouMaz}.
\begin{lemma}
\label{lemTrepeat}
We have the following properties of the endofunctor $T_\alpha$ of $\cO$:\\
(i) The functor $T_\alpha$ is right exact. The left derived functor $\mathcal{L}T_{\alpha}:\cD^-(\cO)\to\cD^-(\cO)$ satisfies $\mathcal{L}_iT_{\alpha}=0$ for $i>1$.\\
(ii) For $M\in\cO$, we have 
\[\begin{cases}T_\alpha M =0,&\mbox{ if $M$ is $\alpha$-finite;}\\
\cL_1 T_\alpha M=0, &\mbox{ if $M$ is $\alpha$-free.}\end{cases}\]
(iii) For any central character $\chi: \mathcal{Z}(\fg)\to\C$, the endofunctor $\mathcal{L}T_{\alpha}\circ\mathcal{R} G_\alpha$ of $\mathcal{D}^+(\cO_\chi)$ is isomorphic to the identity functor.\\
(iv) We have the equivalence of endofunctors on $\cO$: $\cL_1 T_\alpha \cong S_\alpha$ and $\cR_1 G_\alpha\cong \Gamma_\alpha$.
\end{lemma}

Now we derive some further properties of these twisting functors, which will be applied in BBW theory.
\begin{lemma}
\label{prepTGroth}
Consider $\alpha$ simple in $\Delta^+_{\oa}$ and $w\in W$ such that $ws_\alpha > w$, then $T_\alpha$ maps projective modules in $\cO$ to acyclic modules for $T_w$.
\end{lemma}
\begin{proof}
Projective modules in $\cO$ are direct summands of modules induced from projective modules in $\cO_{\oa}$. The claim therefore follows from equation \eqref{eqlemma51} and the corresponding statement for Lie algebras, see e.g. the proof of Corollary 6.2 in~\cite{MR2032059}.
\end{proof}

The following lemma is an application of the principle in the proof of Proposition~3 in \cite{MR2366357}.
\begin{lemma}
\label{IDMaz}
If $M\in\cO$ is $\alpha$-free and $N\in\cO$ is $\alpha$-finite, then we have
$$\Ext_{\cO}^k(T_\alpha M,N)\cong \Ext_{\cO}^{k-1}(M, N).$$ 
\end{lemma}
\begin{proof}
Applying the properties in Lemma \ref{lemTrepeat} implies
\begin{eqnarray*}
\Ext^k_{\cO}(T_\alpha M,N)&\cong& \Hom_{\cD^+(\cO)}(\cL T_\alpha M,N[k])\\
&\cong &\Hom_{\cD^+(\cO)}(M,\cR G_\alpha N[k])\\
&\cong&  \Ext_{\cO}^{k-1}(M, N),
\end{eqnarray*}
which yields the lemma.
\end{proof}

\begin{proposition}
\label{extVermafin}
Consider an arbitrary weight $\mu\in\fh^\ast$, $w\in W$ and $V$ a finite dimensional module in $\cO$. We have
$$\Ext^j_{\cO}(T_w M(\mu), V)=\begin{cases}\Ext^{j-l(w)}_{\cO}(M(\mu),V)&\mbox{if }l(w)\le j\\ 0&\mbox{ if } l(w) > j.\end{cases}$$
\end{proposition}
\begin{proof}
We have $\cL_k T_w M(\mu)=0$ for every $k>0$, this follows from the corresponding property for Lie algebras, see Theorem 2.2 in \cite{MR2032059}, equation \eqref{eqlemma51} and the fact that $\Res M(\mu)$ has a standard filtration in $\cO_{\oa}$.

On the other hand, the property $\cL_kT_w V=\delta_{l(w),k}V$ holds for $l(w)=1$ by Lemma~\ref{lemTrepeat}(ii). The general case follows by induction from this, using the Grothen-dieck spectral sequence, which is well-defined by Lemma \ref{prepTGroth}.

The claim then follows from these two properties as in the proof of Lemma \ref{IDMaz}.
\end{proof}

\begin{lemma}
\label{TG1}
If $M\in \cO$ and $G_\alpha M$ are both $\alpha$-free, then we have $M\cong T_\alpha G_\alpha M$.
\end{lemma}
\begin{proof}
This follows immediately from Lemma \ref{lemTrepeat}.\end{proof}

These results allow to conclude the following two corollaries.
\begin{corollary}
\label{resVerma}
Consider a $\Gamma^+$-generic $\Lambda\in\cP^+$ and $w\in W$ with reduced expression $w=s_{\alpha_1}s_{\alpha_2}\cdots s_{\alpha_k}$, then we have
$$\Res \,G_{s_{\alpha_j}s_{\alpha_{j-1}}\cdots s_{\alpha_1}}\, M_{w\cdot\Lambda}\,\,\cong\,\, \Res\, M_{s_{\alpha_{j+1}}\cdots s_{\alpha_k}\cdot\Lambda}\quad\mbox{for }1\le j\le k.$$
\end{corollary}
\begin{proof}
The PBW theorem implies that we have
$$\Res\, M_\Lambda\cong \bigoplus_{\gamma\in \Gamma^+}M^{\oa}_{\Lambda-\gamma}\quad\mbox{and}$$
\begin{equation}\label{resMM}\Res\, M_{w\cdot\Lambda}\cong \bigoplus_{\gamma\in \Gamma^+}M^{\oa}_{w\cdot\Lambda-\gamma}\cong \bigoplus_{\gamma\in \Gamma^+}M^{\oa}_{w\circ(\Lambda-\gamma)}.\end{equation}
Since the twisting functors (and therefore their adjoints) on category $\cO$ and $\cO_{\oa}$ intertwine the restriction operator, see  equation \eqref{eqlemma51}, it suffices to prove that 
$$G_{s_{\alpha_j}s_{\alpha_{j-1}}\cdots s_{\alpha_1}}^{\oa} M^{\oa}_{w\circ(\Lambda-\gamma)}\cong M^{\oa}_{s_{\alpha_{j+1}}\cdots s_{\alpha_k}\circ(\Lambda-\gamma)}.$$
By Definition \ref{defgeneric}(i) all the weights $\Lambda-\gamma$ are $\fg_{\oa}$-integral dominant. The equation above is therefore standard and follows e.g. from the combination of Theorems 4.1 and 2.3 in \cite{MR2032059}.
\end{proof}

\begin{corollary}
\label{extcomple}
Let $\Lambda$ be a $\Gamma^+$-generic integral dominant weight, $w\in W$ with $l(w)=k$ and $V$ a finite dimensional $\fg$-module, then we have
$$\Ext^k_{\cO}(M_{w\cdot\Lambda},V)\cong \Hom_{\cO}(G_{w^{-1}}M_{w\cdot\Lambda}, V).$$
\end{corollary}
\begin{proof}
Applying the combination of Lemma \ref{TG1} and Corollary \ref{resVerma} iteratively yields $M_{w\cdot\Lambda}\cong T_wG_{w^{-1}}M_{w\cdot\Lambda}$. Applying Proposition \ref{extVermafin} then yields the result.
\end{proof}

We recall the following immediate consequence of the result of Penkov in Theorem~2.2 of \cite{MR1309652}.
\begin{lemma}
\label{genPen}
Consider $\Lambda\in\cP^+$, $\Gamma^+$-generic. For any $w\in W$, the $\fg_{\oa}$-module $\Res L(w\cdot\Lambda)$ is semisimple and its length only depends on $\fg$ and the degree of atypicality of $\Lambda$.
\end{lemma}

We will also need the following estimate on the star action.
\begin{lemma}
\label{lemeststar}
Consider $\mu\in\fh^\ast$, $\widetilde{\Gamma}$-generic, then 
$$w\ast\mu\,\in\, w\cdot\mu-\Gamma^+.$$
\end{lemma}
\begin{proof}
The combination of Lemmata 8.3 and 7.2 in \cite{CouMaz} shows that $w\ast\mu\in w'\circ (\mu-\Gamma^+)$ for some $w'\in W$. From Theorem 8.10 in \cite{CouMaz} we find $w'=w$. The result therefore follows from equation \eqref{MusGS}.
\end{proof}

The following proposition yields important information on the top of the representation $G_{w^{-1}}M_{w\cdot\Lambda}$ for $\Lambda$ generic, see Definition \ref{defgeneric}(iii).
\begin{proposition}
\label{Berncompl}
For $\Lambda$ a generic integral dominant weight and $w\in W$, we have $G_{w^{-1}}M_{w\cdot\Lambda}\tto L_{w^{-1}\ast w\cdot\Lambda}$.
\end{proposition}
\begin{proof}
Define the $\fg$-module $K$ as the kernel of the morphism $M_{w\cdot\Lambda}\tto L_{w\cdot\Lambda}$ and assume $w=s_\alpha w'$ with $l(w')=l(w)-1$. The left exact functor $G_\alpha$ therefore yields an exact sequence
$$G_\alpha M_{w\cdot\Lambda}\to G_{\alpha}L_{w\cdot\Lambda}\to \cR_1G_\alpha K.$$

Equation \eqref{resMM} shows that $M_{w\cdot\Lambda}$ does not have a simple subquotient corresponding to the Weyl chamber  $w'$. In particular, Lemma \ref{lemTrepeat}(iv) implies that $\cR_1G_\alpha K$ does not have such a subquotient  either. Lemma 8.3 in \cite{CouMaz} implies that $L(s_\alpha\ast w\cdot\Lambda)$ is a subquotient of $G_{\alpha}L_{w\cdot\Lambda}$. Lemma \ref{genPen} and equation \eqref{eqlemma51} imply that this is in fact the only subquotient of $G_{\alpha}L_{w\cdot\Lambda}$ in this Weyl chamber. We can conclude that the exact sequence above yields an epimorphism from $G_\alpha M_{w\cdot\Lambda}$ to a module which has a unique simple subquotient corresponding to the Weyl chamber of $w'$, namely $L(s_\alpha\ast w\cdot\Lambda)$. Corollary \ref{resVerma} implies that the only simple modules in the top of $G_\alpha M_{w\cdot\Lambda}$ belong to the Weyl chamber of $w'$, so we obtain $$G_\alpha M_{w\cdot\Lambda}\tto L(s_\alpha\ast w\cdot\Lambda).$$

Since $\Lambda$ is generic, every weight of the form $w'\ast w\cdot \Lambda$ is $\widetilde{\Gamma}$-generic, by Lemma~\ref{lemeststar}. Therefore, the procedure described above can be repeated until finally  the result $G_{w^{-1}}M_{w\cdot\Lambda}\tto L_{w^{-1}\ast w\cdot\Lambda}$ follows.
\end{proof}

\subsection*{Acknowledgement}
The author is a Postdoctoral Fellow of the Research Foundation - Flanders (FWO).\\ The author wishes to express his gratitude to Ruibin Zhang, Volodymyr Mazorchuk and Vera Serganova for many interesting discussions and to the School of Mathematics and Statistics -- The University of Sydney for hospitality during part of this research.

\end{document}